\documentclass[12pt,reqno]{amsart}

\makeatletter
\@namedef{subjclassname@1991}{\textup{2020} Mathematics Subject Classification}
\makeatother

\usepackage{color}
\usepackage{amsmath, amsthm, amssymb, amsxtra}
\usepackage{amsfonts}
\usepackage[utf8]{inputenc}
\usepackage[dvips]{epsfig}
\usepackage{graphicx}
\usepackage[english]{babel}
\usepackage{hyperref}

\usepackage{graphicx}
\usepackage{latexsym}
\usepackage{mathtools}
\usepackage{esint}
\usepackage{float}

\theoremstyle{plain}
\newtheorem{thm}{Theorem}[section]
\newtheorem{cor}[thm]{Corollary}
\newtheorem{lem}[thm]{Lemma}
\newtheorem{prop}[thm]{Proposition}

\theoremstyle{definition}
\newtheorem{defn}[thm]{Definition}

\theoremstyle{remark}
\newtheorem{obs}[thm]{Remark}

\numberwithin{equation}{section}

\newcommand{\average}{{\mathchoice {\kern1ex\vcenter{\hrule height.4pt
width 6pt depth0pt} \kern-9.7pt} {\kern1ex\vcenter{\hrule
height.4pt width 4.3pt depth0pt} \kern-7pt} {} {} }}

\newcommand{\R}{\mathbb R}
\newcommand{\Z}{\mathbb Z}

\newcommand{\p}{\partial}

\begin{document}

\title[Optimal regularity for supercritical parabolic obstacle problems]{Optimal regularity for supercritical\\ parabolic obstacle problems}
\author{Xavier Ros-Oton}
\address{ICREA, Pg.\ Llu\'is Companys 23, 08010 Barcelona, Spain \&\newline\indent
Universitat de Barcelona, Departament de Matem\`atiques i Inform\`atica, Gran Via de les Corts Catalanes 585, 08007 Barcelona, Spain.}
\email{xros@icrea.cat}
\author{Dami\`a Torres-Latorre}
\address{Universitat de Barcelona, Departament de Matem\`atiques i Inform\`atica, Gran Via de les Corts Catalanes 585, 08007 Barcelona, Spain.}
\email{\tt dtorres-latorre@ub.edu}

\begin{abstract}
We study the obstacle problem for parabolic operators of the type $\p_t + L$,
where $L$ is an elliptic integro-differential operator of order $2s$, such as $(-\Delta)^s$, in the supercritical regime $s \in (0,\frac{1}{2})$. The best result in this context was due to Caffarelli and Figalli, who established the $C^{1,s}_x$ regularity of solutions for the case $L = (-\Delta)^s$, the same regularity as in the elliptic setting.

Here we prove for the first time that solutions are actually \textit{more} regular than in the elliptic case. More precisely, we show that they are $C^{1,1}$ in space and time, and that this is optimal. We also deduce the $C^{1,\alpha}$ regularity of the free boundary. Moreover, at all free boundary points $(x_0,t_0)$, we establish the following expansion:
$$(u - \varphi)(x_0+x,t_0+t) = c_0(t - a\cdot x)_+^2 + O(t^{2+\alpha}+|x|^{2+\alpha}),$$
with $c_0 > 0$, $\alpha > 0$ and $a \in \R^n$.
\end{abstract}

\thanks{XR and DT have received funding from the European Research Council (ERC) under the Grant Agreement No 801867.}
\maketitle

\section{Introduction}\label{sect:intro}
\addtocontents{toc}{\protect\setcounter{tocdepth}{1}}

The aim of this paper is to study the parabolic obstacle problem
\begin{equation}\label{eq:prev_problem}
\left\{
\begin{array}{rclll}
     \min\{\p_tu+Lu, u-\varphi\} & = & 0 & \text{in} & \R^n \times (0,T)\\
     u(\cdot,0) & = & \varphi & \text{in} & \R^n,
\end{array}\right.
\end{equation}
for nonlocal operators of the form \begin{equation}\label{eq:operator}
    Lu(x) = \int_{\R^n}\big(u(x) - u(x+y)\big)K(y)\mathrm{d}y.
\end{equation}
The kernel $K$ is even and satisfies the uniform ellipticity condition
\begin{equation}\label{eq:operator_elliptic}
    \lambda|y|^{-n-2s} \leq K(y) \leq \Lambda|y|^{-n-2s}, \quad K(y) = K(-y),
\end{equation}
for some $0 < \lambda \leq \Lambda$ and $s \in (0,1)$. We define the contact set $\{u = \varphi\}$ and the free boundary $\p\{u > \varphi\}$.

We are mostly interested on studying the \emph{supercritical} case, $s \in (0,\frac{1}{2})$, in which the higher order term is the time derivative instead of the diffussion term. This will give rise to a somewhat unusual approach to the problem, as well as some surprising results.


Nonlocal operators arise naturally when one considers jump-diffussion processes. One of the most classical motivations is the modelling of stock prices, because the nonlocality takes into account the possible large fluctuations of the market. In the trading of options on financial markets, the valuation of American options is an optimal stopping problem. Thus, when the underlying asset price follows a jump-diffussion process, we are led naturally to the parabolic obstacle problem (\ref{eq:prev_problem}); see \cite{CT04, CF13} for details. These models were first introduced in the 1970s by Nobel prize winner R. Merton \cite{Mer76}, and have been used for many years \cite{Sch03,CT04,OS07}.

\subsection{The elliptic case}
From the mathematical point of view, elliptic and parabolic equations involving jump-diffussion operators have been an active and successful field of research in the past two decades, coming from PDE and from Probability.

The first nonlocal operator of this type to be studied was the fractional Laplacian,
$$(-\Delta)^su(x) = c_{n,s}\int_{\R^n}\frac{u(x)-u(x+y)}{|y|^{n+2s}}\mathrm{d}y,$$
and problems involving it can be treated as lower-dimensional problems for local operators via the Caffarelli-Silvestre extension\footnote{Actually, the paper \cite{CS07} was motivated by the study of the fractional obstacle problem in \cite{CSS08,Sil07}.} \cite{CS07}.


The elliptic obstacle problem,
$$\min\{Lu, u - \varphi\} = 0 \quad \text{in} \quad \Omega,$$
was studied for the case of $L = (-\Delta)^s$ by Caffarelli, Salsa and Silvestre using the extension and local arguments in \cite{CSS08}. Using a new Almgren-type monotonicity formula, they established the optimal $C^{1,s}$ regularity of solutions. Furthermore, they proved the following dichotomy at the free boundary points:

\begin{itemize}
    \item Either $x_0$ is a \textit{regular} free boundary point, and
    $$cr^{1+s} \leq \sup\limits_{B_r(x_0)}(u - \varphi) \leq Cr^{1+s} \quad \forall r \in (0,r_0),$$
    where $c > 0$.
    \item Or, if $x_0$ is not regular, it is called \textit{singular} and then
    $$0 \leq \sup\limits_{B_r(x_0)}(u-\varphi) \leq Cr^2 \quad \forall r \in (0,r_0).$$
\end{itemize}
Moreover, they also proved that the regular points are an open subset of the free boundary and that they are locally a $C^{1,\alpha}$ manifold.

It is important to notice that, in contrast with the classical case $s = 1$, there is no nondegeneracy property of the solutions, i.e. at singular points we may have $\sup\limits_{B_r(x_0)}(u-\varphi) \asymp r^k$ with $k \gg 1$.

The regularity of the free boundary and related questions have been widely investigated in the recent years by several authors. See \cite{GP09,BFR18b,FS18,CSV20,FJ21,SY21} for more information on the singular points, \cite{KPS15,DS16,JN17,KRS19} for higher regularity of the free boundaries, \cite{CRS17,AR20} for more general elliptic operators and \cite{PP15,GPPS17,FR18,Kuk21} for operators with drift.


\subsection{The parabolic case}
Much less is known about the parabolic case (\ref{eq:prev_problem}). Notice that the problem now depends strongly on the value of $s$: in the subcritical case $s \in (\frac{1}{2},1)$, the higher order term is the nonlocal operator, in the critical case $s = \frac{1}{2}$, both $\p_t$ and $L$ are of order one, and in the supercritical case $s \in (0,\frac{1}{2})$, the higher order term is the time derivative.

The first result in this direction was the regularity of the solutions in the case $L = (-\Delta)^s$ due to Caffarelli and Figalli \cite{CF13}, where they established the $C^{1,s}$ regularity in $x$ for all $s \in (0,1)$, and conjectured it to be optimal. They also established the $C^{1,\beta}$ regularity in $t$, with $\beta = \frac{1-s}{2s} - 0^+$ when $s \geq 1/3$, and that $u_t$ is log-Lipschitz in $t$ when $s < 1/3$. 
Their proof uses crucially the extension problem for the fractional Laplacian and the $C^{1,s}_x$ regularity is established via a new monotonicity formula for such problem.

Then, the regularity of the free boundary near regular points was established in the subcritical case, $s \in (\frac{1}{2},1)$, by Barrios, Figalli and the first author in \cite{BFR18}, where they establish a dichotomy for the free boundary points completely analogous to the elliptic case (in particular, $C^{1,s}_x$ regularity is optimal). One the main difficulties in \cite{BFR18} was to establish a classification of blow-ups in a context where Almgren-type monotonicity formulas are not available.

More recently, Borrin and Marcon established the quasi-optimal regularity of solutions for the subcritical case, $s \in (\frac{1}{2},1)$, for a more general equation allowing lower order terms \cite{BM21}.

Despite these developments, in the supercritical case $s \in (0,\frac{1}{2})$ the only known result was the regularity of the solutions for the fractional Laplacian proved in \cite{CF13}. Quite surprisingly, we prove here that this was not optimal, and that solutions are $C^{1,1}$ in $x$ and $t$.

\subsection{Main results}

Our main results are the following. We first establish the optimal regularity of the solutions.
\begin{thm}\label{thm:1.1}
Let $s \in (0,\frac{1}{2})$, and let $u$ be the solution of (\ref{eq:prev_problem}) with $L$ an operator satisfying (\ref{eq:operator}) and (\ref{eq:operator_elliptic}), and $\varphi \in C^{2,1}_c(\R^n)$.

Then, $u$ is Lipschitz in $\R^n\times[0,T]$ and
$$u \in C^{1,1}(\R^n\times(0,T]),$$
i.e., the solution $u$ is globally\footnote{Here we mean that for all $t_0 > 0$, $u \in C^{1,1}(\R^n\times[t_0,T])$.} $C^{1,1}$ in $x$ and $t$.
\end{thm}

It is important to notice that because of the initial condition in (\ref{eq:prev_problem}), the solution $u(x,t)$ can never be a solution of the elliptic problem; this is why solutions might be more regular than in the elliptic case. Notice also, though, that our solution $u$ to (\ref{eq:prev_problem}) always converges as $T \rightarrow \infty$ to a solution to the elliptic problem. For this reason, we cannot expect to get a uniform $C^{1,1}$ bound in $\R^n\times(0,\infty)$.

Our proof is completely different from \cite{CF13}, and actually it is mainly based on barriers, comparison principles, and the supercritical scaling of the equation. 
In particular, we do not use any monotonicity formula.
This allows us not only to get the optimal $C^{1,1}$ regularity for the fractional Laplacian but also to extend the result to general integro-differential operators.

Then, we prove the global $C^{1,\alpha}$ regularity of the free boundary.
\begin{thm}\label{thm:1.2}
Let $s \in (0,\frac{1}{2})$, and let $u$ be the solution of (\ref{eq:prev_problem}) with $L$ an operator satisfying (\ref{eq:operator}) and (\ref{eq:operator_elliptic}), and $\varphi \in C^{2,1}_c(\R^n)$. Then,
\begin{itemize}
    \item The free boundary $\p\{u > \varphi\}$ is a $C^{1,\alpha}$ graph in the $t$ direction,
    $$\p\{u > \varphi\} = \{t = \Gamma(x)\}$$
    with $\Gamma \in C^{1,\alpha}$ and $\alpha > 0$.
    \item If $(x_0,t_0)$ is any free boundary point, the solution admits an expansion
    \begin{equation}\label{eq:expansion}
    (u-\varphi)(x_0+x,t_0+t) = c_0(t-a\cdot x)_+^2 + O(t^{2+\alpha}+|x|^{2+\alpha}),
    \end{equation}
    where $c_0 > 0$, $\alpha > 0$ and $a \in \R^n$.
\end{itemize}
\end{thm}

To have that \textit{all} free boundary points have the same expansion is a very uncommon result in the context of obstacle problems, and it contrasts notably with the parabolic subcritical and the elliptic obstacle problems. 
Moreover, the blow-up techniques that are always used to study free boundaries appeared ineffective here.
Our proof of Theorem \ref{thm:1.2} uses Theorem \ref{thm:1.1} and the fact that $L$ has order $2s < 1$ to gain further regularity instead.

This global regularity result allows us to define regular and singular points \textit{a posteriori} in a very simple way: we say that a free boundary point $(x_0,t_0)$ is regular if the vector $a$ in the expansion (\ref{eq:expansion}) is not zero, and is singular if $a = 0$. 

Finally, as a consequence of Theorem \ref{thm:1.2}, we deduce that the free boundary is $C^{1,\alpha}$ in the $x$ direction near regular points, and that singular points are in some sense scarce.

\begin{thm}\label{thm:1.3}
Let $s \in (0,\frac{1}{2})$, and let $u$ be the solution of (\ref{eq:prev_problem}) with $L$ an operator satisfying (\ref{eq:operator}) and (\ref{eq:operator_elliptic}), and $\varphi \in C^{2,1}_c(\R^n)$. Then,
\begin{itemize}
    \item The set of regular free boundary points is an open subset of $\p\{u > \varphi\}$.
    \item If $(x_0,t_0)$ is a regular free boundary point, the free boundary $\p\{u > \varphi\}$ is locally a $C^{1,\alpha}$ graph in the $x_i$ direction for some $i \in \{1,\ldots,n\}$,
    $$\p\{u > \varphi\}\cap B_r(x_0,t_0) = \{x_i = F(x_1,\ldots,x_{i-1},x_{i+1},\ldots,x_n,t)\},$$
    with $F \in C^{1,\alpha}$, $\alpha > 0$ and $r > 0$.
    \item Let $\Sigma_t$ be the set of singular free boundary points $(x_0,t_0)$ with $t_0 = t$. Then,
    $$\mathcal{H}^{n-1}(\Sigma_t) = 0 \quad \text{for almost every} \quad t \in (0,T).$$
\end{itemize}
\end{thm}

This problem is very different than the rest of elliptic and parabolic free boundary problems. Notice how Theorem \ref{thm:1.2} establishes a regularity result common to regular and singular free boundary points, which deeply contrasts with how these problems were approached until now. Besides, the fact that the free boundary is globally a $C^{1,\alpha}$ graph in the $t$ direction could also be true in the classical ($s = 1$) case, but is not known in the latter setting.

\begin{obs}
There is more literature available for the related (but not equivalent) obstacle problem with operator $(\p_t-\Delta)^s$. It appears when one considers the parabolic thin obstacle problem ($s = \frac{1}{2}$) or the parabolic thin obstacle problem with a weight. In this setting, the diffussion term is always the highest order term and thus the scaling is always subcritical. For more information on the topic, see \cite{AC10,BSZ17,DGPT17,ACM18,Shi20} and references therein.
\end{obs}

\subsection{Plan of the paper}
The paper is organized as follows.

In Section \ref{sect:preliminaries} we prove a comparison principle and the semiconvexity of solutions. Then, in Section \ref{sect:C1} we prove that the solutions to (\ref{eq:prev_problem}) are $C^1$, and in Section \ref{sect:C11}, we show that the optimal regularity is $C^{1,1}$. Finally, Section \ref{sect:FB} is devoted to proving the $C^{1,\alpha}$ regularity of the free boundary and Theorem \ref{thm:1.3}.

Besides, we include some technical tools in two appendices. Appendix \ref{sect:appendix_heat} includes several regularity and growth estimates for the linear nonlocal parabolic equation, and Appendix \ref{sect:appendix_penaliz} is a discussion about the penalized obstacle problem.

\section{Preliminaries and semiconvexity}\label{sect:preliminaries}
In this Section we give some basic definitions and prove some basic results that will be used later on.

Given any solution $u$ of (\ref{eq:prev_problem}), we define
$$v(x,t) = u(x,t) - \varphi(x).$$
Notice that $\p_t u = \p_t v$. Let $B_r(x_0)$ be the ball of radius $r$ and center $x_0$ in $\R^n$, and let $Q_r(x_0,t_0)$ be the following parabolic cylinders:
$$Q_r(x_0,t_0) = B_r(x_0) \times (t_0 - r^{2s}, t_0 + r^{2s})$$
When the balls or cylinders are centered at the origin we will just write $B_r := B_r(0)$ and $Q_r := Q_r(0,0)$.

We will denote $\nabla := \nabla_x$, and we will write $\nabla_{x,t}$ when we refer to the gradient in all variables.

We will also define the following weighted $L^1$ norm:
$$\|u\|_{L^1_s} = \|u\|_{L^1_s(\R^n)} := \int_{\R^n}\frac{|u(x)|}{1+|x|^{n+2s}}\mathrm{d}x$$
and the corresponding weighted Lebesgue space
$$L^1_s(\R^n) := \{f : \R^n \to \R, f\text{ measurable}, \|f\|_{L^1_s} < +\infty\}.$$

\subsection{Basic tools}

We recall some standard tools for elliptic and parabolic PDE that are useful to deal with problem (\ref{eq:prev_problem}). Let us start with the comparison principle.
\begin{thm}\label{thm:comparison_obstacle}
Let $L$ be a nonlocal operator satisfying (\ref{eq:operator}) and (\ref{eq:operator_elliptic}), let $\varphi$ and $\psi$ be uniformly Lipschitz and bounded, and let $u$ and $v$ be the solutions of the following parabolic problems:
$$\left\{
\begin{array}{rclll}
     \min\{\p_tu+Lu, u-\varphi\} & = & 0 & \text{in} & \R^n \times (0,T)\\
     u(\cdot,0) & = & \varphi & \text{in} & \R^n,
\end{array}\right.$$
$$\left\{
\begin{array}{rclll}
     \min\{\p_tv+Lv, v-\psi\} & = & 0 & \text{in} & \R^n \times (0,T)\\
     v(\cdot,0) & = & \psi & \text{in} & \R^n.
\end{array}\right.$$
Assume additionally that $\varphi \leq \psi$. Then, $u \leq v$ in $\R^n\times(0,T)$.
\end{thm}

To prove it, we use the penalization method. This approximation technique is based in considering the solutions to the obstacle problem as the limit of the solutions to the following parabolic problem
\begin{equation}\label{eq:penalized}
    \left\{\begin{array}{rclll}
    \p_tu^\varepsilon + Lu^\varepsilon & = & \beta_\varepsilon(u^\varepsilon - \varphi) & \text{in} & \R^n\times(0,T)\\
    u^\varepsilon(\cdot,0) & = & \varphi + \sqrt{\varepsilon},
    \end{array}\right.
\end{equation}
where $\beta_\varepsilon(z) = e^{-z/\varepsilon}$.

\begin{lem}\label{lem:penalization}
Let $L$ be an operator satisfying (\ref{eq:operator}) and (\ref{eq:operator_elliptic}), let $\varphi \in C^{2,1}_c(\R^n)$ and let $u^\varepsilon$ be the solution of (\ref{eq:penalized}).

Then, $u^\varepsilon \rightarrow u^0$ as $\varepsilon \rightarrow 0$ locally uniformly, where $u^0$ is the solution of (\ref{eq:prev_problem}).
\end{lem}

We give the proof in Appendix \ref{sect:appendix_penaliz}. Using this technique, we can now proceed.

\begin{proof}[Proof of Theorem \ref{thm:comparison_obstacle}]
It suffices to write $u$ and $v$ as the limits of the penalized versions of the respective problems, and then apply Lemma \ref{lem:comparison_beta}.
\end{proof}

The following observation is based in the strong maximum principle and will be important in our discussion.

\begin{lem}\label{lem:derivada_positiva}
Let $u$ be a solution of (\ref{eq:prev_problem}) with $L$ an operator satisfying (\ref{eq:operator}) and (\ref{eq:operator_elliptic}), and $\varphi \in C^{0,1}_c(\R^n)$. Then,
$$u_t > 0 \quad \text{in} \quad \{u > \varphi\}.$$
\end{lem}

\begin{proof}
First, we see that $u$ is nondecreasing in $t$. Consider the function $\tilde{u}(x,t) = u(x,t+\delta)$, $\delta > 0$. Then, $\tilde{u}$ is clearly also a solution of $\min\{(\p_t+L)\tilde{u},\tilde{u} - \varphi\} = 0$, and $\tilde{u}(\cdot,0) = u(\cdot,\delta) \geq u(\cdot,0) = \varphi$. Hence, $\tilde{u}$ is a supersolution of (\ref{eq:prev_problem}), and thus $\tilde{u} \geq u$. This yields $u(x,t+\delta) \geq u(x,t)$ for all $x$, $t$ and $\delta > 0$.

Let $w = u_t$. Differentiating (\ref{eq:prev_problem}), we have
$$\p_t w + Lw = 0 \quad \text{in} \quad \{u > \varphi\}.$$
We also know that $w \geq 0$ because $u$ is nondecreasing in time. Suppose $w = 0$ at $(x,t) \in \{u > \varphi\}$. Then, by the strong maximum principle, $w \equiv 0$ in all the connected component of $(x,t)$. In particular, $w = 0$ in the segment $\{x\} \times [0,t]$ because each point in the segment belongs either to the contact set or to the connected component of $(x,t)$ in $\{u > \varphi\}$. Hence, $u(x,t) = u(x,0) = \varphi(x)$, contradicting $(x,t) \in \{u > \varphi\}$. Therefore, $w > 0$ in  $\{u > \varphi\}$.
\end{proof}

\subsection{Semiconvexity}
An essential property of the solutions is that they are semiconvex, see \cite[Lemma 2.1]{BFR18} for the case $L = (-\Delta)^s$ with $s > \frac{1}{2}$. Here we can use the same strategy to prove it.

\begin{prop}\label{prop:semiconvex}
Let $s \in (0,\frac{1}{2})$, and let $u$ be a solution of (\ref{eq:prev_problem}), with $L$ an operator satisfying (\ref{eq:operator}) and (\ref{eq:operator_elliptic}), and $\varphi \in C^{2,1}_c(\R^n)$. Then, $u$ is semiconvex, i.e., for all unit vectors $e$ in $x,t$, $\p_{ee}u \geq -\hat{C}$, with a uniform bound that depends only on $\varphi$, $n$, $s$ and the ellipticity constants.
\end{prop}

\begin{obs}
The assumption $s \in (0,\frac{1}{2})$ can be substituted by the more general $s \in (0,1)$ and $\varphi \in C^{\max\{2,4s+\varepsilon\}}_c$ for some small $\varepsilon > 0$.
\end{obs}

\begin{proof}[Proof of Proposition \ref{prop:semiconvex}]
Using Lemma \ref{lem:penalization}, we can write $u$ as the limit of solutions to the penalized problem (\ref{eq:penalized}). Since the locally uniform limit of uniformly semiconvex functions is semiconvex, we only need to prove it for the approximations $u^\varepsilon$.

First, we use Lemma \ref{lem:diff_penaliz} and notice that $\beta_\varepsilon'' \geq 0$ to obtain
$$\p_tu^\varepsilon_{\nu\nu} + Lu^\varepsilon_{\nu\nu} \geq \beta_\varepsilon'(u^\varepsilon - \varphi)(u^\varepsilon_{\nu\nu} - \varphi_{\nu\nu}),$$
for any unit vector $\nu \in \R^n\times\R$, and also
\begin{align*}
    u^\varepsilon_t(\cdot,0) &= e^{-1/\sqrt{\varepsilon}} - L\varphi,\\
    u^\varepsilon_{tt}(\cdot,0) &= L^2\varphi - \frac{1}{\varepsilon}e^{-1/\sqrt{\varepsilon}}(e^{-1/\sqrt{\varepsilon}}-L\varphi).
\end{align*}

Define $C_0 := \|u^\varepsilon_{\nu\nu}(\cdot,0)\|_{L^\infty(\R^n)}$. Then,
\begin{align*}
    C_0 &\leq \|D^2_xu^\varepsilon(\cdot,0)\|_{L^\infty(\R^n)} + \|\nabla u^\varepsilon_t(\cdot,0)\|_{L^\infty(\R^n)} + \|u^\varepsilon_{tt}(\cdot,0)\|_{L^\infty(\R^n)}\\
    &\leq \|D^2\varphi\|_{L^\infty(\R^n)} + \|\nabla L\varphi\|_{L^\infty(\R^n)} + \|L^2\varphi - \frac{1}{\varepsilon}e^{-1/\sqrt{\varepsilon}}(e^{-1/\sqrt{\varepsilon}}-L\varphi)\|_{L^\infty(\R^n)}\\
    &\leq \|D^2\varphi\|_{L^\infty(\R^n)} + \|\nabla L\varphi\|_{L^\infty(\R^n)} + \|L^2\varphi\|_{L^\infty(\R^n)} + C\varepsilon + \|L\varphi\|_{L^\infty(\R^n)}\\
    &\leq C\|\varphi\|_{C^{1,1}(\R^n)} + C\varepsilon.
\end{align*}

Using again that $\beta_\varepsilon' \leq 0$, it follows that $\beta_\varepsilon'(u^\varepsilon - \varphi)(u^\varepsilon_{\nu\nu}+C_0) \geq 0$ whenever $u^\varepsilon_{\nu\nu} + C_0 \leq 0$. Hence, $w := \min\{0,u^\varepsilon_{\nu\nu} + C_0\}$ satisfies
$$\p_tw + Lw \geq 0 \quad \text{in} \quad \R^n\times(0,T).$$

Finally, $w \equiv 0$ at $t = 0$ by construction, hence, by the maximum principle, $w \equiv 0$ everywhere, i.e. $u^\varepsilon_{\nu\nu} \geq -C_0$. Since this constant does not depend on $\varepsilon$, we can pass to the limit to get the desired result.
\end{proof}

\section{$C^1$ regularity of solutions}\label{sect:C1}
Here we prove that solutions $u$ to the problem (\ref{eq:prev_problem}) are globally $C^1$ in $x$ and $t$. This was already known in the case of $L = (-\Delta)^s$ thanks to \cite{CF13}; here we prove it in a different way for our general class of operators (\ref{eq:operator}). The first step is to prove global Lipschitz regularity.

Notice that we already know that $u$ is Lipschitz because it is globally bounded and semiconvex, but we provide a simple proof to obtain the optimal Lipschitz constant under the minimal requirements for $\varphi$.
\begin{prop}\label{prop:global_lip}
Let $s \in (0,\frac{1}{2})$, and let $u$ be a viscosity solution of (\ref{eq:prev_problem}) with $L$ an operator satisfying (\ref{eq:operator}) and (\ref{eq:operator_elliptic}), and $\varphi \in C^{0,1}_c(\R^n)$. Then, $u$ is globally Lipschitz,
$$\|\nabla u\|_{L^\infty(\R^n\times(0,T))} \leq \|\varphi\|_{C^{0,1}(\R^n)} \quad \text{and} \quad \|u_t\|_{L^\infty(\R^n\times(0,T))} \leq C\|\varphi\|_{C^{0,1}(\R^n)},$$
where $C$ depends only on the dimension, $s$ and the ellipticity constants.
\end{prop}

\begin{proof}
First of all, $\|u\|_{L^\infty(\R^n\times(0,T))} \leq \|\varphi\|_{L^\infty(\R^n\times(0,T))}$ by Theorem \ref{thm:comparison_obstacle}.

We will treat Lipschitz regularity in $x$ and $t$ separately. For spatial regularity, observe that for every $h \in \R^n$, the function $w_h(x,t) := u(x+h,t) + \|\varphi\|_{C^{0,1}}|h|$ is a solution of
\begin{equation*}
\left\{
\begin{array}{rclll}
     \min\{\p_tw_h+Lw_h, w_h-\varphi_h\} & = & 0 & \text{in} & \R^n \times (0,T]\\
     w_h(\cdot,0) & = & \varphi_h & \text{in} & \R^n,
\end{array}\right.
\end{equation*}
with $\varphi_h(x) = \varphi(x+h) + \|\varphi\|_{C^{0,1}}|h| \geq \varphi$. Then, by Theorem \ref{thm:comparison_obstacle}, $u \leq w_h$ for all $h$, and it follows that
$$u(x,t) \leq u(x+h,t) + \|\varphi\|_{C^{0,1}}|h| \quad \Rightarrow \quad \frac{u(x,t) - u(x+h,t)}{|h|} \leq \|\varphi\|_{C^{0,1}}.$$
Since $x$ and $h$ are arbitrary, the Lipschitz regularity follows.

On the other hand, concerning $u_t$, it is zero in the interior of the contact set, and outside of it $u_t = -Lu$. Moreover, since $u$ is continuous, the contact set is closed and we can estimate the Lipschitz character of $u$ in the $t$ direction knowing it outside of the contact set. Hence, $\|u_t\|_{L^\infty(\R^n\times(0,T))} \leq \|Lu\|_{L^\infty(\R^n\times(0,T))}$. Then, we can compute $Lu$. We omit the time dependence to unclutter the notation.
\begin{align*}
    |Lu(x)| &= \left|\int_{\R^n}(u(x) - u(x+y))K(y)\mathrm{d}y\right|\\
    &\leq  \left|\int_{B_1}(u(x) - u(x+y))K(y)\mathrm{d}y\right|
    + \left|\int_{B_1^c}(u(x) - u(x+y))K(y)\mathrm{d}y\right|\\
    &\leq \int_{B_1}\|\nabla u\|_{L^\infty(\R^n\times(0,T))}|y|K(y)\mathrm{d}y + \int_{B_1^c}2\|u\|_{L^\infty(\R^n\times(0,T))}K(y)\mathrm{d}{y}\\
    &\leq C_1\|\nabla u\|_{L^\infty(\R^n\times(0,T))} + C_2\|u\|_{L^\infty(\R^n\times(0,T))} \leq C\|\varphi\|_{C^{0,1}(\R^n)}.
\end{align*}
Here we used that $K(y) \leq \Lambda|y|^{-n-2s}$ and $s < \frac{1}{2}$, so that $K(y)$ is integrable at infinity and $|y|K(y)$ is integrable near the origin, and finally we applied the previous estimates for $\|\nabla u\|_{L^\infty}$ and $\|u\|_{L^\infty}$ in terms of $\|\varphi\|_{C^{0,1}}$.
\end{proof}

Then, we improve the regularity up until $C^{1,\alpha}$ in $t$ and $C^1$ in $x$. We start with the time regularity.

\begin{prop}\label{prop:t_C1alpha}
Let $s \in (0,\frac{1}{2})$, and let $u$ be the solution of (\ref{eq:prev_problem}) with $L$ an operator satisfying (\ref{eq:operator}) and (\ref{eq:operator_elliptic}), and $\varphi \in C^{0,1}_c(\R^n)$. Then, $u_t \in C^{\alpha}$ and
$$[u_t]_{C^\alpha(\R^n\times(0,T))} \leq C\|\varphi\|_{C^{0,1}(\R^n)},$$
where $\alpha = 1 - 2s > 0$ and $C$ depends only on the dimension, $s$ and the ellipticity constants. Moreover, we have
$$u_t = (Lu)^- \quad \text{in} \quad \R^n\times(0,T).$$
\end{prop}

\begin{proof}
Let us prove the following estimates for $Lu$ to begin. We prove the spatial regularity first, omitting the time dependence for simplicity of reading.
\begin{align*}
    |Lu(x_1) - Lu(x_2)| &= \left|\int_{\R^n}(u(x_1)-u(x_2)-u(x_1+y)+u(x_2+y))K(y)\mathrm{d}y\right|\\
    &\leq \int_{B_r}\left(|u(x_1) - u(x_1+y)| + |u(x_2) - u(x_2 + y)|\right)K(y)\mathrm{d}y\\
    &+ \int_{B_r^c}\left(|u(x_1) - u(x_2)| + |u(x_1+y) - u(x_2+y)|\right)K(y)\mathrm{d}y\\
    &\leq \left(\int_{B_r}2|y|K(y)\mathrm{d}y + \int_{B_r^c}2|x_1-x_2|K(y)\mathrm{d}y\right)\|\nabla u\|_{L^\infty(\R^n\times(0,T))}\\
    &\leq C(r^{1-2s} + |x_1-x_2|r^{-2s})\|\nabla u\|_{L^\infty(\R^n\times(0,T))}\\
    &\leq C\|\varphi\|_{C^{0,1}(\R^n\times(0,T))}|x_1-x_2|^{1-2s}.
\end{align*}
In the last steps we used that $|K(y)| \leq \Lambda|y|^{-n-2s}$, with $s \in (0,\frac{1}{2})$, we chose $r = |x_1 - x_2|$ and we used the estimate from Proposition \ref{prop:global_lip}.

Then, we prove temporal regularity:

\begin{align*}
    &|Lu(x,t_1) - Lu(x,t_2)|\\
    &= \left|\int_{\R^n}(u(x,t_1)-u(x,t_2)-u(x+y,t_1)+u(x+y,t_2))K(y)\mathrm{d}y\right|\\
    &\leq \int_{B_r}\left(|u(x,t_1) - u(x+y,t_1)| + |u(x,t_2) - u(x + y,t_2)|\right)K(y)\mathrm{d}y\\
    &+ \int_{B_r^c}\left(|u(x,t_1) - u(x,t_2)| + |u(x+y,t_1) - u(x+y,t_2)|\right)K(y)\mathrm{d}y\\
    &\leq \int_{B_r}2|y|K(y)\|\nabla u\|_{L^\infty(\R^n\times(0,T))}\mathrm{d}y + \int_{B_r^c}2|t_1-t_2|K(y)\|u_t\|_{L^\infty(\R^n\times(0,T))}\mathrm{d}y\\
    &\leq C\|\varphi\|_{C^{0,1}(\R^n\times(0,T))}|t_1-t_2|^{1-2s}.
\end{align*}
Here $r = |t_1 - t_2|$ and the rest of the estimates are used analogously.

Hence, $[Lu]_{C^\alpha(\R^n\times(0,T))} \leq C\|\varphi\|_{C^{0,1}(\R^n)}$. In particular, $Lu$ is continuous. Then, recall that $u_t + Lu = 0$ in the set $\{u > \varphi\}$. Moreover, by Lemma \ref{lem:derivada_positiva}, $u_t > 0$ in this set, and therefore $Lu < 0$.

In the interior of the contact set, however, $u(x,t) \equiv \varphi(x)$ and $u_t \equiv 0$. Moreover, $u_t + Lu \geq 0$, and it follows that $Lu \geq 0$ in the interior of the contact set. 

By continuity of $Lu$, $Lu = 0$ on the free boundary. Then, $u_t = 0$ on the free boundary as well.

We deduce that 
$$u_t = (Lu)^-$$
and thus $[u_t]_{C^\alpha(\R^n\times(0,T))} \leq [Lu]_{C^\alpha(\R^n\times(0,T))}$, as wanted.
\end{proof}

Then, we continue with the regularity in $x$. First, we need the following estimate, analogous to the elliptic estimate \cite[Lemma 2.3]{CRS17}.

\begin{lem}\label{lem:x_regularity}
Let $s \in (0,1)$. There exist constants $\tau \in (0,s)$ and $\delta > 0$ such that the following holds.

Let $v$ be a globally Lipschitz solution of
\begin{equation*}
\left\{\begin{array}{rclll}
v & \geq & 0 & \text{in} & \R^n\times(-1,0]\\
\p_{\nu\nu}v & \geq & -\delta & \text{in} & Q_2 \cap \{t \leq 0\}, \quad \text{for all} \quad \nu \in \mathbb{S}^{n-1}\\
(\p_t + L)(v-T_hv) & \leq & \delta|h| & \text{in} & \{v > 0\}\cap Q_2\cap \{t \leq 0\},
\end{array}\right.
\end{equation*}
where $T_h$ is the translation operator defined for any $h$ in the $x$ directions, i.e. $T_hv(x,t) = v(x+h,t)$, and $L$ is a nonlocal operator satisfying (\ref{eq:operator}) and (\ref{eq:operator_elliptic}).

Assume that $v(0,t) = 0$ for $t \leq 0$ and that $\sup\limits_{Q_R}|\nabla v| \leq R^\tau$ for all $R \geq 1$. Then,
$$\sup\limits_{B_r\times(-r^{2s},0]}|\nabla v| \leq 2r^\tau,$$
for all $r > 0$. The constants $\tau$ and $\delta$ depend only on the dimension, $s$ and the ellipticity constants.
\end{lem}

\begin{proof}
Let $W_r = B_r \times (-r^{2s},0]$ be the \textit{past cyilinders at the origin}.

We define
$$\theta(r) := \sup\limits_{r' \geq r}(r')^{-\tau}\sup\limits_{W_{r'}}|\nabla v|.$$

Notice that $\theta(r) \leq 1$ for $r \geq 1$ because $\sup\limits_{W_R}|\nabla v| \leq \sup\limits_{Q_R}|\nabla v| \leq R^\tau$ for $R \geq 1$. The result we aim to prove is equivalent to showing $\theta(r) \leq 2$ for all $r \in (0,1)$. Observe also that $\theta$ is nonincreasing by definition.

Assume by contradiction that $\theta(r) > 2$ for some $r$. Then, by construction there exists $r_0 \in (r,1)$ such that
$$\theta(r_0) \geq r_0^{-\tau}\sup_{W_{r_0}}|\nabla v| \geq (1-\varepsilon)\theta(r) \geq (1-\varepsilon)\theta(r_0) \geq \frac{3}{2},$$
where $\varepsilon > 0$ is to be chosen later.

Then, we define the scaling
$$v_0(x,t) := \frac{v(r_0x,r_0^{2s}t)}{\theta(r_0)r_0^{1+\tau}}.$$
Let $\tau \in (0,s)$. Then, the rescaled function satisfies
\begin{equation*}
\left\{\begin{array}{rclllll}
v_0 & \geq & 0 &&& \text{in} & \R^n\times(-2r_0^{-2s},0]\\
\p_{\nu\nu} v_0 & \geq & -r_0^{2-1-\tau}\delta & \geq & -\delta & \text{in} & Q_{2/r_0}\\
(\p_t + \tilde L)(v_0-T_hv_0) & \leq & r_0^{2s-1-\tau}\delta|r_0h| & \leq & \delta|h| & \text{in} & \{v_0 > 0\}\cap Q_{2/r_0},
\end{array}\right.
\end{equation*}
where $\tilde L$ is the corresponding nonlocal operator with the appropriate scaled kernel, and it has the same ellipticity constants. Notice that $\|\nabla v_0\|_{L^\infty(W_1)} \leq 1$ by construction.

Moreover, by the definition of $\theta$ and $r_0$, for all $R \geq 1$ the following estimates hold:
$$1 - \varepsilon \leq \sup\limits_{|h| \leq \frac{1}{4}}\sup\limits_{W_1}\frac{v_0(x,t) - v_0(x+h,t)}{|h|} \quad \text{and} \quad \sup\limits_{|h| \leq \frac{1}{4}}\sup\limits_{W_R}\frac{v_0 - T_hv_0}{|h|} \leq (R + \frac{1}{4})^\tau.$$

Let $\eta \in C^2_c(Q_{3/2})$ with $\eta \equiv 1$ in $Q_1$ and $0 \leq \eta \leq 1$. Then,
$$\sup\limits_{|h| \leq \frac{1}{4}}\sup\limits_{W_1}\left(\frac{v_0 - T_hv_0}{|h|} + 3\varepsilon\eta\right) \geq 1 + 2\varepsilon.$$

Notice that if $\tau > 0$ is small enough,
$$\sup\limits_{|h| \leq \frac{1}{4}}\sup\limits_{W_3}\frac{v_0 - T_hv_0}{|h|} \leq (3+\frac{1}{4})^\tau < 1 + \varepsilon.$$

Then, we can choose $h_0 \in B_{1/4}$ such that
$$M := \max\limits_{W_{3/2}}\left(\frac{v_0 - T_{h_0}v_0}{|h_0|} + 3\varepsilon\eta\right) \geq 1 + \varepsilon,$$
and the maximum is attained at a point $(x_0,t_0)$ where $\eta(x_0,t_0) > 0$.

Define
$$w := \frac{v_0 - T_{h_0}v_0}{|h_0|}.$$

By construction, $w + 3\varepsilon\eta \leq M$ in $W_{3/2}$ and in $W_3 \setminus W_{3/2}$. Therefore, $w + 3\varepsilon\eta \leq M$ in $Q_3 \cap \{t \leq 0\}$. Besides, $v_0(x_0,t_0) > 0$ because if not $w(x_0,t_0) < 0$ and then $w+3\varepsilon\eta < 1 + \varepsilon$.

Now we evaluate the equation at $(x_0,t_0)$ to obtain a contradiction.

On the one hand, since $(x_0,t_0)$ is a maximum of $w + 3\varepsilon\eta$, and $(x_0,t_0)$ is either an interior point of $W_{3/2}$ or a point in $B_{3/2}\times\{0\}$,
$$\p_t(w+3\varepsilon\eta) \geq 0.$$

On the other hand, we can use the semiconvexity of $v_0$, together with $v_0(0,t) = 0$ for $t \leq 0$ to obtain a lower bound for $\tilde Lw$. Let $e = \frac{h_0}{|h_0|}$ and $k = |h_0|$. Then, for $x \in B_1$ and omitting the dependence on $t$,
$$v_0(x) \leq \frac{kv_0(0) + |x|v_0\left(x + k\frac{x}{|x|}\right)}{k + |x|} + \frac{k\delta}{2}|x|^2 \leq v_0\left(x + k\frac{x}{|x|}\right) + \delta,$$
using that $|x| < 1$ and $k < 1$. Then, combining this fact with the definition of $w$,
$$w(x) = \frac{v_0(x) - v_0(x+ke)}{k} \leq \frac{v_0\left(x + k\frac{x}{|x|}\right) - v_0(x + ke)}{k} + \delta \leq \left|\frac{x}{|x|} + e\right| + \delta,$$
for all $x \in B_1$, where we also used that $\|\nabla v_0\|_{L^\infty(W_1)} \leq 1$. In particular, $w(x,t) < \frac{1}{2}$ for all $t \leq 0$ when $\delta < \frac{1}{4}$ and
$$x \in C_e := \left\{x \in B_1 : \left|\frac{x}{|x|} + e\right| < \frac{1}{4}\right\}.$$

Using that $M \geq 1 + \varepsilon$ and $w < 1 + \varepsilon$ in $W_3$,
$$1 - 2\varepsilon \leq w(x_0,t_0) < 1 + \varepsilon.$$

Moreover, $w + 3\varepsilon\eta$ has a maximum at $(x_0,t_0)$ (global in $B_3\times\{t_0\}$), and hence 
$$w(x_0,t_0) - w(x,t_0) \geq -3\varepsilon|D^2\eta(x_0,t_0)|\frac{|x-x_0|^2}{2} = -C\varepsilon|x-x_0|^2,$$
for all $x \in B_3$.

Let us now compute $\tilde Lw$ at the point $(x_0,t_0)$. Using the previous estimates,

\begin{align*}
\tilde Lw(x_0,t_0) = &\int_{\R^n}(w(x_0,t_0) - w(x_0+y,t_0))K(y)\mathrm{d}y\\
\geq &\,\lambda\int_{\R^n}(w(x_0,t_0) - w(x_0+y,t_0))_+|y|^{-n-2s}\mathrm{d}y\\
&- \Lambda\int_{\R^n}(w(x_0,t_0) - w(x_0+y,t_0))_-|y|^{-n-2s}\mathrm{d}y\\
\geq &\,\lambda\int_{C_e - x_0}(\frac{1}{2} - 2\varepsilon)|y|^{-n-2s}\mathrm{d}y - \Lambda\int_{B_{3/2}}C\varepsilon|y|^2|y|^{-n-2s}\mathrm{d}y\\
&-\Lambda\int_{B_{3/2}^c}((|y|+\frac{3}{2})^\tau - 1 + 2\varepsilon)|y|^{-n-2s}\mathrm{d}y\\
\geq &\,c - C\varepsilon - \Lambda\int_{B_{3/2}^c}((|y|+\frac{3}{2})^\tau - 1)|y|^{-n-2s}\mathrm{d}y \geq c - C\varepsilon,
\end{align*}
where in the last step we choose $\tau > 0$ even smaller if needed to absorb the integral into the $C\varepsilon$ term.

Finally, 
$$(\p_t + \tilde L)w(x_0,t_0) \geq -3\varepsilon\eta_t(x_0,t_0) + c - C\varepsilon > \delta,$$
choosing small enough $\varepsilon$ and $\delta$, reaching a contradiction. Hence, $\theta(r) \leq 2$ for all $r \in (0,1)$, as we wanted to prove.
\end{proof}

Now we can apply Lemma \ref{lem:x_regularity} to obtain $C^1$ regularity.

\begin{prop}\label{prop:grad_cont}
Let $s \in (0,\frac{1}{2})$, and let $u$ be the solution of (\ref{eq:prev_problem}) with $L$ an operator satisfying (\ref{eq:operator}) and (\ref{eq:operator_elliptic}), and $\varphi \in C^{2,1}_c(\R^n)$. Then, $\nabla u \in C(\R^n\times(0,T))$. In particular, $u \in C^1(\R^n\times(0,T))$.
\end{prop}

\begin{proof}
First, by Proposition \ref{prop:t_C1alpha}, $u_t$ is already continuous, and by Proposition \ref{prop:global_lip}, $\nabla u$ is globally defined in $L^\infty$. We will prove that it is continuous at every point.

In the interior of the contact set, $u(x,t) \equiv \varphi(x) \in C^1$, and in the interior of $\{u > \varphi\}$, we can use interior estimates (Proposition \ref{prop:interior_alphabeta}) to see that $u$ is $C^1$.

Therefore, we only need to work with the points on the free boundary. Assume without loss of generality that the origin is a free boundary point, and we will prove that $\nabla u$ is continuous at it.

Let $v = u - \varphi$. After a scaling and a translation, we can apply Lemma \ref{lem:x_regularity} to obtain
$$\sup \limits_{B_R(x_0)\times(t_0-R^{2s},t_0]}|\nabla v| \leq CR^\tau,$$
for all $R \geq 0$. The constant $C$ here depends only on $\varphi$, the dimension, $s$ and the ellipticity constants.

We distinguish two cases:

\textit{Case 1.} If \textit{the free boundary continues to the future}, more precisely, for all $\rho \in (0,r)$, there exists $t_\rho > 0$ such that
$$\{v = 0\}\cap(B_\rho\times\{t_\rho\}) \neq \emptyset,$$
it follows that for all $t \in (0,t_\rho)$, $\{v = 0\}\cap(B_\rho\times\{t\}) \neq \emptyset$, because $u_t \geq 0$ and therefore the contact set shrinks in time.

Let $\delta \in (0,r)$. Let $|x| < \delta$, and $t < t_\delta$ as defined above. Then, there exists $x' \in B_\delta$ such that $(x',t)$ belongs to the contact set, and it follows that
$$|\nabla v(x,t)| \leq C|x-x'|^\tau \leq C(2\delta)^\tau.$$

Then, letting $\delta \rightarrow 0$, we obtain a sequence of neighbourhoods of the origin where $|\nabla v| \leq C(2\delta)^\tau$, and hence $\nabla v$ vanishes continuously at $(0,0)$.

\textit{Case 2.} If \textit{the free boundary ends at the origin}, there exists some $r_0 > 0$ such that for all $t > 0$, $v > 0$ in $B_{r_0}\times\{t\}$. Assume after a scaling that $r_0 = 1$ (notice that $L$ may change but the ellipticity constants will be the same). We will prove that the limit of $v_i$ is zero as it approaches the origin. If we approach from the past, then $(0,-t)$ belongs to the contact set for all $t > 0$, and we can use the same argument that in Case 1.

To consider approaching the origin from the future, recall that $u$ solves $u_t = (Lu)^-$ globally, hence, we can consider $u$ a solution of the nonlocal heat equation with right hand side
$$(\p_t + L)u = (Lu)^+ \quad \text{in} \quad \R^n\times(0,T')$$
and apply Duhamel's formula at $(x,t)$ with $x \in B_{1/2}$ and $t \in (0,\frac{1}{2})$, to get
$$u(x,t) = \int_{\R^n}p_t(x-y)u(y,0)\mathrm{d}y + \int_{0}^t\int_{\R^n}p_{t-\zeta}(x-y)(Lu)^+(y,\zeta)\mathrm{d}y\mathrm{d}\zeta,$$
where $p_t(x)$ is the fundamental solution for this particular operator (see Theorem \ref{thm:kernel_estimates}). Then, differentiating with respect to $x_i$ and using that $p_t \in C^\infty$ and $u$ is Lipschitz,
$$u_i(x,t) = \int_{\R^n}p_t(x-y)u_i(y,0)\mathrm{d}y + \int_{0}^t\int_{\R^n}\p_ip_{t-\zeta}(x-y)(Lu)^+(y,\zeta)\mathrm{d}y\mathrm{d}\zeta.$$

Now let us estimate both integrals separately. For the first one, we will use that $|u_i(y,0)| \leq C|y|^\tau$ by Lemma \ref{lem:x_regularity}, as well as $|p_t(x)| \leq C\min\{t^{-\frac{n}{2s}}, t|x|^{-n-2s}\}$ by Theorem \ref{thm:kernel_estimates}.
\begin{align*}
&\left|\int_{\R^n}p_t(x-y)u_i(y,0)\mathrm{d}y\right| \lesssim \int_{\R^n}\min\{1,|y|^\tau\}\min\left\{t^{-\frac{n}{2s}},\frac{t}{|x-y|^{n+2s}}\right\}\mathrm{d}y\\
&\lesssim \int_{B_{t^{\frac{1}{2s}}/2}(x)}t^{-\frac{n}{2s}}|y|^\tau\mathrm{d}y + \int_{B_{1/2}(x)\setminus B_{t^{\frac{1}{2s}}/2}(x)}\frac{t|y|^\tau}{|x-y|^{n+2s}}\mathrm{d}y + \int_{B_{1/2}^c(x)}\frac{t}{|x-y|^{n+2s}}\mathrm{d}y\\
&\leq t^{-\frac{n}{2s}}|x+t^{\frac{1}{2s}}|^\tau|B_{t^{\frac{1}{2s}}}| + t\int_{B_{1/2}\setminus B_{t^{\frac{1}{2s}}}}|y|^{-n-2s}|x+y|^\tau\mathrm{d}y + t\int_{B_{1/2}^c}|y|^{-n-2s}\mathrm{d}y\\
&\lesssim |x+t^{\frac{1}{2s}}|^\tau + t(t^{\frac{1}{2s}})^{-2s}|x|^\tau + t(t^{\frac{1}{2s}})^{\tau - 2s} + t \lesssim t^{\frac{\tau}{2s}} + |x|^\tau.
\end{align*}

For the second integral, we will use that $Lu$ is bounded because $u$ is Lipschitz, $Lu \leq 0$ in $B_{1/2}(x) \subset B_1$, as well as $Lu \leq 0$ outside of the support of the obstacle $\varphi$. Let $R$ big enough such that $\operatorname{supp}\varphi \subset B_R$. Then, by Corollary \ref{cor:grad_kernel},
\begin{align*}
\left|\int_{0}^t\int_{\R^n}\p_ip_{t-\zeta}(x-y)(Lu)^+(y,\zeta)\mathrm{d}y\mathrm{d}\zeta\right| &\lesssim \int_0^t\int_{B_R \setminus B_{1/2}(x)}|\nabla p_{t-\zeta}(x-y)|\mathrm{d}y\mathrm{d}\zeta\\
&\lesssim \int_0^t\int_{B_R \setminus B_{1/2}(x)} 1\mathrm{d}y\mathrm{d}\zeta \lesssim t.
\end{align*}

Therefore, $|u_i(x,t)| \lesssim t^{\frac{\tau}{2s}} + |x|^\tau$ for $t > 0$, and it converges to zero as it approaches the origin from the future, concluding that $\nabla u$ is continuous in $x$ and $t$ at that point.
\end{proof}

\section{Optimal $C^{1,1}$ regularity}\label{sect:C11}
In this section, we establish the optimal $C^{1,1}$ regularity of solutions. First, we prove that the free boundary moves at a positive \textit{speed}.

\begin{prop}\label{prop:FB_es_mou}
Let $s \in (0,\frac{1}{2})$, and let $u$ be the solution of (\ref{eq:prev_problem}) with $L$ an operator satisfying (\ref{eq:operator}) and (\ref{eq:operator_elliptic}), and $\varphi \in C^{2,1}_c(\R^n)$\footnote{The compactness of the support is a technical condition needed for the proof of this proposition but it does not seem crucial for the problem.}. Let $v = u - \varphi$, and let $0 < t_1 < t_2 < T$. Then,
$$|\nabla v| \leq C v_t \quad \text{in} \quad \R^n\times[t_1,t_2],$$
for some positive $C$, depending only on $t_1$, $t_2$, $\varphi$, the dimension, $s$ and the ellipticity constants.

Moreover, the free boundary is the graph of a Lipschitz function $\{t = \Gamma(x)\}$ in $\R^n\times(t_1,t_2)$, with the same Lipschitz constant $C$.
\end{prop}

To prove this proposition, we will use the following positivity lemma, see \cite[Lemma 6.2]{CRS17} for the elliptic version.
\begin{lem}\label{lem:nondeg}
Let $E \subset Q_1$ be compact, let $L$ be an operator satisfying (\ref{eq:operator}) and (\ref{eq:operator_elliptic}), and let $w \in C(Q_1) \cap C^1(Q_1\setminus E)$ satisfying
\begin{equation*}
\left\{\begin{array}{rclll}
|\p_tw + Lw| & \leq & \varepsilon & \text{in} & Q_1 \setminus E\\
w & = & 0 & \text{in} & E\\
w & \geq & -\varepsilon & \text{in} & E^c,
\end{array}\right.
\end{equation*}
in the viscosity sense, and also
$$\int_{\R^n}\frac{w^+(x,t)}{1+|x|^{n+2s}}\mathrm{d}x \geq 1 \quad \text{for all} \quad t \in [-1,1].$$
Then,
$$w \geq 0 \quad \text{in} \quad \overline{Q_{1/2}}.$$

The constant $\varepsilon > 0$ depends only on $s$, the dimension and the ellipticity constants.
\end{lem}

\begin{proof}
Let $\psi \in C^\infty_c(Q_{3/4})$, with $\psi \equiv 1$ in $Q_{1/2}$ and $0 \leq \psi \leq 1$. We proceed by contradiction. Suppose the lemma does not hold. Then, for some $c > 0$, the function
$$\psi_{\varepsilon,c} = -c - \varepsilon + \varepsilon\psi$$
touches $w$ from below in $(x_0,t_0) \in Q_{3/4}$. Moreover, $(x_0,t_0) \in E^c$ because $w(x_0,t_0) < 0$, so $(x_0,t_0) \in Q_1 \setminus E$.

Now we compute $(\p_t+L)w(x_0,t_0)$ to obtain a contradiction. By the definition of $(x_0,t_0)$, $w - \psi_{\varepsilon,c}$ attains a global minimum there. Thus,
\begin{align*}
(\p_t+L)(w-\psi_{\varepsilon,c})(x,t) &= L(w-\psi_{\varepsilon,c})(x,t)\\
&= -\int_{\R^n}(w(x+y,t) - \psi_{\varepsilon,c}(x+y,t))K(y)\mathrm{d}y\\
&\leq -\lambda\int_{\R^n}w^+(x+y,t)|y|^{-n-2s}\mathrm{d}y\\
&\leq -\lambda\int_{\R^n}\frac{w^+(y,t)}{|y-x|^{n+2s}}\mathrm{d}y \leq -C\lambda,
\end{align*}
using that $\psi_{\varepsilon,c} < 0$ and that $|y-x|^{n+2s} \leq C(1+|y|^{n+2s})$ for any $x \in B_{3/4}$, with $C$ depending only on $n+2s$.

On the other hand,
$$(\p_t+L)(w-\psi_{\varepsilon,c})(x,t) = (\p_t+L)w(x,t) - (\p_t+L)\psi_{\varepsilon,c} \leq \varepsilon + \varepsilon\|(\p_t+L)\psi\|_{L^\infty(Q_{3/4})},$$
and choosing $\varepsilon$ small enough we get a contradiction.
\end{proof}

Using this lemma we are now able to prove that the free boundary \textit{moves at all values of} $t$, i.e., it is a Lipschitz graph in the $t$ direction.

\begin{proof}[Proof of Proposition \ref{prop:FB_es_mou}]
We will prove the inequality for any directional derivative $v_i$ instead of the gradient. The result follows as a consequence.

Let $R \geq \max\{1,T^{\frac{1}{2s}}\}$ be such that $\operatorname{supp} \varphi \subset B_R$ and let $P > 0$ large, to be chosen later. Consider the set $A = \overline{B_1}(3Re_1)\times[\frac{t_1}{2},\frac{t_2+T}{2}]$. Then, by construction, $A \subset \{v > 0\}$, and from Lemma \ref{lem:derivada_positiva} and compactness, it follows that $v_t \geq a > 0$ in~$A$.

Let $r > 0$ such that for all $(x_0,t_0) \in B_{PR}\times[t_1,t_2]$, $Q_r(x_0,t_0) \subset \R^n\times[\frac{t_1}{2},\frac{t_2+T}{2}]$. We will use a rescaled Lemma \ref{lem:nondeg} in $Q_r(x_0,t_0)$ with a suitable linear combination 
$$w = Mv_t - mv_i$$
with some positive $M$ and $m$ to be chosen later.

First, let $E$ be the contact set. Then, $w \geq -m\|v_i\|_{L^\infty(\R^n\times(0,T))} \geq -2m\|\varphi\|_{C^{0,1}(\R^n)}$ in the whole space by Proposition \ref{prop:global_lip}. Moreover, in $E^c$ we have
$$|(\p_t+L)w| = m|(\p_t+L)v_i| = m|-L\varphi_i| \leq m\|\varphi\|_{C^{1,1}(\R^n)} \quad \text{in} \quad E^c.$$

On the other hand, for all $t \in [t_0 - r^{2s}, t_0 + r^{2s}]$,
\begin{align*}
&\int_{\R^n}\frac{w^+(x,t)}{1+|x-x_0|^{n+2s}}\mathrm{d}x \geq \int_{B_1(3Re_1)}\frac{w^+(x,t)}{1+|x-x_0|^{n+2s}}\mathrm{d}x \geq\\
&\quad \int_{B_1(3Re_1)}\frac{Ma - m\|v_i\|_{L^\infty(\R^n\times(0,T))}}{1 + |x-x_0|^{n+2s}}\mathrm{d}x \geq \frac{(Ma - m\|\varphi\|_{C^{0,1}(\R^n)})|B_1|}{1+(PR+3R+1)^{n+2s}}.
\end{align*}

Then, choosing $m$ small enough and $M$ big enough suffices to be able to apply Lemma \ref{lem:nondeg}, and these constants depend only $n$, $s$, $\lambda$, $\Lambda$, $R$ and $\varphi$. Therefore, $w \geq 0$ in $B_{PR}\times[t_1,t_2]$.

Finally, outside of $B_{PR}$, we will use a barrier argument. Since $v_t > 0$ in the set $(\overline{B_{PR/2}}\setminus B_R)\times[0,T]$, by compactness we can choose $M$ and $m$ such that $w(\cdot,0) \geq 0$ in $B_{PR}$ and also $w \geq m$ in $(\overline{B_{PR/2}}\setminus B_R)\times[0,T]$.

Let $\tilde{w} = w + m(1+2\|\varphi\|_{C^{0,1}(\R^n)})\chi_{B_R}$. Now, since 
$$w \geq mv_i \geq - m\|v_i\|_{L^\infty(\R^n\times[0,T])} \geq -2m\|\varphi\|_{C^{0,1}(\R^n)},$$
$\tilde{w} \geq m$ in $B_{PR/2}\times[0,T)$. On the other hand, $v = u - \varphi$ is identically zero at $t = 0$ and $v_t(\cdot,0) = -L\varphi > 0$ outside of the support of $\varphi$, and hence $\tilde{w}(\cdot,0) \geq 0$ in $B_{PR}^c$.

To apply the comparison principle, we also need to compute the right hand side for $x \in B_{PR}^c$. Using that $u$ is a solution of the nonlocal heat equation,
$$(\p_t+L)\tilde{w} = (\p_t+L)(w + m(1+2\|\varphi\|_{C^{0,1}(\R^n)})\chi_{B_R}) = mL[\varphi_i + (1+2\|\varphi\|_{C^{0,1}(\R^n)})\chi_{B_R}],$$
and since the expression inside of the brackets is supported in $B_R$, for all $x$ such that $|x| \geq PR$,
\begin{align*}
    |(\p_t+L)\tilde{w}| &\leq C'mR^n\|\varphi_i + (1+2\|\varphi\|_{C^{0,1}(\R^n)})\chi_{B_R}\|_{L^\infty(B_R\times[0,T))}(|x|-R)^{-n-2s}\\
    &\leq CmR^n|x|^{-n-2s} \quad \text{in} \quad B_{PR}^c\times[t_1,t_2],
\end{align*}
where $C$ depends only on $n$, $s$, $\lambda$, $\Lambda$ and $\varphi$.

Let now $\psi$ be defined as the solution of
$$\left\{\begin{array}{rclll}
(\p_t+L)\psi & = & \left[(\p_t+L)\tilde{w}\right]\chi_{B_{PR}^c} & \text{in} & \R^n\times(0,T)\\
\psi & = & \frac{m}{2}\chi_{B_{PR/2}} & \text{on} & \R^n\times\{t = 0\}\\
\end{array}\right.$$

Then, $|(\p_t+L)\psi| \leq CmR^n(PR)^{-n-2s} = CmP^{-n-2s}R^{-2s}$, and it follows that $(\p_t+L)(\psi - CmP^{-n-2s}R^{-2s}t) \leq 0$. Therefore, since it is a subsolution for the nonlocal heat equation, applying the comparison principle\footnote{Here, $\psi$ can be defined with the Duhamel formula and the heat kernel introduced in Theorem \ref{thm:kernel_estimates}, and the comparison principle follows from the positivity of the heat kernel.} with a constant we deduce\linebreak $\psi-CmP^{-n-2s}R^{-2s}t \leq \frac{m}{2}$ in $\R^n\times(0,T)$, and in particular $\psi \leq \frac{m}{2} + CmP^{-n-2s}R^{-2s}T$. Choosing $P$ large enough, $\psi \leq m$ in $\R^n\times(0,T)$.

Now, we apply the comparison principle again. Notice that $\psi \leq \tilde{w}$ at $t = 0$ by construction, and that $(\p_t+L)\psi = (\p_t+L)\tilde{w}$ for all $(x,t) \in B_{PR}^c\times(0,T)$. Furthermore, $\psi \leq m \leq \tilde{w}$ in $B_{PR}\times(0,T)$. Therefore, $\psi \leq \tilde{w}$ in $\R^n\times(0,T)$.

Finally, let
$$\tilde{\psi}(x,t) = \frac{2}{m|B_{1}|}\psi\left(\frac{PR}{2}x,\left(\frac{PR}{2}\right)^{2s}t\right).$$
Then, $\tilde{\psi}(\cdot,0) = |B_{1}|^{-1}\chi_{B_{1}}$, so it is positive, supported in $B_1$ and $\|\tilde{\psi}(\cdot,0)\|_{L^1(B_1)} = 1$. Moreover,
$$|(\p_t+L)\tilde{\psi}| \leq \frac{2CmR^n}{m|B_{1}|}\left(\frac{PR}{2}\right)^{2s}\left|\frac{PR}{2}x\right|^{-n-2s}\chi_{B_1^c} \leq C'P^{-n}|x|^{-n-2s}\chi_{B_1^c},$$
and if we take $P$ large enough such that $C'P^{-n} < \delta$, from Proposition \ref{prop:barrier} we get that $\tilde{\psi} \geq 0$ in $B_2^c\times(0,T(PR/2)^{-2s})$\footnote{Here we need to choose $P$ large enough to have $T(PR/2)^{-2s} < 1$.}. Then, $\tilde{w} \geq 0$ in $B_2^c\times(0,T)$, and since $\tilde{w} = w$ in $B_{PR}^c\times(0,T)$ we obtain $w \geq 0$ in $B_{PR}^c\times(0,T)$, as we wanted to prove.

From the inequality $|\nabla v| \leq Cv_t$, it follows that the free boundary is a Lipschitz graph in the $t$ direction with constant $C$.
\end{proof}

Once we know that the free boundary is a Lipschitz graph in the direction of $t$, we can use barriers to gain insight on the boundary behaviour of $v_t$. We will prove first a Hopf-type estimate in the $t$ direction. Here we use crucially the fact that the diffussion is supercritical, i.e. $s < \frac{1}{2}$.

\begin{prop}\label{prop:hopf}
Let $s \in (0,\frac{1}{2})$, and let $u$ be the solution of (\ref{eq:prev_problem}) with $L$ an operator satisfying (\ref{eq:operator}) and (\ref{eq:operator_elliptic}), and $\varphi \in C^{2,1}_c(\R^n)$. Let $v = u - \varphi$, and let $0 < t_1 < t_2 < T$. Then, there exists $c_0 > 0$ such that for all free boundary points $(x_0,t_0) \in \R^n\times[t_1,t_2]$,
$$v_t(x_0,t_0+t) \geq c_0t \quad \text{for} \quad t \in (0,\delta),$$
where $c_0$ and $\delta$ are small positive constants depending only on $t_1$, $t_2$, $T$, $\varphi$, the dimension, $s$ and the ellipticity constants.
\end{prop}

\begin{proof}
Let $R \geq 1$ such that $\operatorname{supp} \varphi \subset B_R$. Then, consider the compact set 
$$A = \overline{B_1}(3Re_1)\times\left[\frac{t_1}{2},\frac{t_2+T}{2}\right].$$
Then, by construction, $A \subset \{v > 0\}$, and from Lemma \ref{lem:derivada_positiva} and compactness, it follows that $v_t \geq a > 0$ in $A$.

By Proposition \ref{prop:FB_es_mou}, there exists $C_0$ such that $|\nabla v| \leq C_0v_t$ in $\R^n\times[\frac{t_1}{2},\frac{t_2+T}{2}]$. Assume without loss of generality that $C_0 \geq 1$.

Now, there exists $r > 0$ such that for all $(x,t) \in \R^n\times[t_1,t_2]$, 
$$Q_r(x,t) \subset \R^n\times\left(\frac{t_1}{2},\frac{t_2+T}{2}\right).$$

Let $(x_0,t_0)$ be a free boundary point with $t_0 \in [t_1,t_2]$, and define the cone
$$\mathcal{C} = \{t_0 + 2C_0|x - x_0| < t < t_0 + r^{2s}\} \subset \R^n\times\left(\frac{t_1}{2},\frac{t_2+T}{2}\right).$$
Since $C_0$ is also the Lipschitz constant of the free boundary in $\R^n\times(\frac{t_1}{2},\frac{t_2+T}{2})$, $\mathcal{C}$ is entirely \textit{above} the free boundary, and $v > 0$ in $\mathcal{C}$. Then, it follows from Lemma \ref{lem:derivada_positiva} that $v_t > 0$ in $\mathcal{C}$ as well.

With this information, we can construct a subsolution in $\mathcal{C}$ to compare with $v_t$. Let us assume after a translation that $(x_0,t_0) = (0,0)$. Let $w$ defined in $\R^n\times[0, r^{2s}]$ as follows:
$$w(x,t) = c_0(t - 2C_0|x|)_+ + a\chi_{\tilde{A}}(x,t) = c_0(t - 2C_0|x|)_+ + a\chi_{\overline{B_1}(3Re_1 - x_0)}(x),$$
with $c_0 > 0$ to be chosen later.

Then, we need to check that $(\p_t + L)w \leq 0$ in $\mathcal{C}$ and that $w \leq v_t$ in $(\R^n\times(0,r))\setminus \mathcal{C}$. The latter follows by construction, because for any $(x,t) \in \R^n\times(0,r)$ that does not belong to $\mathcal{C}$, $t - 2C_0|x| < 0$ and then $w \equiv a\chi_{\overline{B_1}(3Re_1 - x_0)}(x)$ in the relevant set. Thus, recalling that $v_t \geq a$ in $A$, $w \leq v_t$ outside of the cone.

To check that $w$ is a subsolution in $\mathcal{C}$, first notice that $w_t = c_0$ inside the cone. Then,
\begin{align*}
    Lw(x,t) &\leq c_0 \|L(t-2C_0|x|)_+\|_{L^\infty(\R^n\times(0,r))} + a(L\chi_{\overline{B_1}(3Re_1 - x_0)})(x)\\
    &\leq C_1C_0c_0 - a\int_{\R^n}\chi_{\overline{B_1}(3Re_1 - x_0)}(y)K(y)\mathrm{d}y \leq C_1C_0c_0 - \frac{a\lambda|B_1|}{(4R+1)^{n+2s}},
\end{align*}
where we used that $|x_0| < R$, and it follows that
$$(\p_t+L)w \leq c_0 + C_1C_0c_0 - C_2$$
and then choosing $c_0$ small enough suffices to have $(\p_t + L)w \leq 0$.

Finally, by the comparison principle\footnote{Here, $v_t$ and $w$ are classical solutions and the comparison principle follows from the standard pointwise bounds. We shall use this feature again in subsequent arguments.}, $v_t \geq w$ in $\mathcal{C}$, and in particular $v_t(0,t) \geq c_0t$ for $t \in [0,r^{2s})$, and undoing the translation,
$$v_t(x_0,t_0+t) \geq c_0t \quad \text{for} \quad t \in (0,r^{2s}),$$
and for all $(x_0,t_0) \in \p\{u > \varphi\}\cap(\R^n\times[t_1,t_2])$, as we wanted to prove.
\end{proof}

Integrating the lower bound for $v_t$, we can obtain a quadratic nondegeneracy of $v$ in the $t$ direction.

\begin{cor}\label{cor:nondeg}
Let $s \in (0,\frac{1}{2})$, and let $u$ be the solution of (\ref{eq:prev_problem}) with $L$ an\linebreak operator satisfying (\ref{eq:operator}) and (\ref{eq:operator_elliptic}), and $\varphi \in C^{2,1}_c(\R^n)$. Let $v = u - \varphi$, and let ${(x_0,t_0) \subset \R^n\times[t_1,t_2]}$ be a free boundary point. Then, there exists $c_0 > 0$ such that
$$v(x_0,t_0+r) \geq c_0r^2$$
for all $r \in (0,\delta)$, where $c_0$ and $\delta$ are positive and depend only on $\varphi$, $t_1$, $t_2$, $T$, $s$, the dimension and the ellipticity constants.
\end{cor}

\begin{proof}
Use Proposition \ref{prop:hopf} to see that $v_t(x_0,t_0+r) \geq c_0r$ for all $r \in (0,\delta)$. Then, since $v \in C^1$, we can recover the value of $v$ integrating $v_t$ and therefore we get $v(x_0,t_0+r) \geq v(x_0,t_0) + c_0r^2/2 = c_0r^2/2$. Finally rename $c_0/2$ as $c_0$.
\end{proof}

The counterpart is an upper bound for the growth of $v_t$. Much like the Hopf-type estimate can be proved with a subsolution taking advantage of a future cone of positivity, the anti-Hopf-type estimate is proved with a supersolution that takes advantage of a past cone in the contact set. Again, here we use crucially that the diffussion is supercritical.

\begin{prop}\label{prop:anti-hopf}
Let $s \in (0,\frac{1}{2})$, and let $u$ be the solution of (\ref{eq:prev_problem}) with $L$ an operator satisfying (\ref{eq:operator}) and (\ref{eq:operator_elliptic}), and $\varphi \in C^{2,1}_c(\R^n)$. Let $v = u - \varphi$, and let $0 < t_1 < t_2 < T$. Then, there exists $M > 0$ such that for all free boundary points $(x_0,t_0) \in \R^n\times[t_1,t_2]$,
$$v_t(x_0,t_0 + t) \leq Mt \quad \text{for all} \quad t > 0,$$
where $M$ depends only on $\varphi$, $t_1$, $t_2$, $T$, $s$, the dimension and the ellipticity constants.
\end{prop}

\begin{proof}
By Proposition \ref{prop:FB_es_mou}, there exists $C_0$ such that $|\nabla v| \leq C_0v_t$ in $\R^n\times[\frac{t_1}{2},\frac{t_2+T}{2}]$. Assume without loss of generality that $C_0 \geq 1$.

Now, there exists $r > 0$ such that for all $(x,t) \in \R^n\times[t_1,t_2]$, 
$$Q_r(x,t) \subset \R^n\times\left(\frac{t_1}{2},\frac{t_2+T}{2}\right).$$

Let $(x_0,t_0)$ be a free boundary point with $t_0 \in [t_1,t_2]$, and define the cone
$$\mathcal{C} = \{t_0 - r^{2s} < t < t_0 - 2C_0|x - x_0|\} \subset \R^n\times\left(\frac{t_1}{2},\frac{t_2+T}{2}\right).$$
Notice that this cone is \textit{backwards}, whereas the cone defined in the proof of Proposition \ref{prop:hopf} was \textit{forward}. Since $C_0$ is also the Lipschitz constant of the free boundary in $\R^n\times(\frac{t_1}{2},\frac{t_2+T}{2})$, $\mathcal{C}$ is entirely \textit{below} the free boundary, and then $v_t \equiv 0$ in $\mathcal{C}$.

Assume after a translation that $(x_0,t_0) = (0,0)$. Now, we want to construct a supersolution in $$\Omega_\rho = B_{\rho}\times(-\rho,\rho) \setminus \mathcal{C},$$
with $\rho \in (0,r)$ to be chosen later.

To do so, we introduce the auxiliary function $h(x,t) := \min\{4C_0+1,(t+|x|)_+\}$. First, we notice $\p_th \equiv 1$ in $\{h > 0\}\cap Q_1$ and estimate $Lh$ as follows.
$$\|Lh\|_{L^\infty(\R^n\times\R)} \leq C_1\|h\|_{C^{0,1}(\R^n\times\R)} = C_1.$$

Let now $h_\rho(x,t) = h(4C_0\rho^{-1}x, \rho^{-1}t)$. By the scaling of the equation (notice that the bound on $Lh$ depends on the ellipticity constants but not on the particular operator),
$$(\p_t+L)h_\rho \geq \rho^{-1} - C_1(4C_0)^{2s}\rho^{-2s} \geq 0 \quad \text{in} \quad \Omega_\rho,$$
provided that $\rho$ is small enough. Notice that $\rho$ depends only on $t_1$, $t_2$, $T$, the dimension, $s$ and the ellipticity constants.

Finally, let us check that there exists $M > 0$ such that $v_t \leq Mh_\rho$ in $\Omega_\rho$. To do so, we will check that $v_t \leq Mh_\rho$ in the parabolic boundary of $\Omega_\rho$. Indeed, $v_t = 0 \leq Mh_\rho$ in $\mathcal{C}$ for any positive $M$. 

On the other hand, if we choose $M = \|v_t\|_{L^\infty(\R^n\times(0,T))}$, for all $t \in [-\rho,\rho]$ and $x \not\in B_\rho$, 
$$h_\rho(x,t) = \min\{1, \rho^{-1}(t+4C_0|x|)_+\} \geq \min\{1, \rho^{-1}(-\rho+4C_0\rho)_+\} = 1,$$
and for all $x \in \overline{B_\rho\times(-\rho,\rho)\setminus\mathcal{C}}$, $|x|\geq \frac{\rho}{2C_0}$, and therefore
$$h_\rho(x,-\rho) = \min\{1, \rho^{-1}(-\rho + 4C_0|x|)_+\} \geq \min\{1, (-1+2)_+\} = 1.$$
Hence,
$$v_t \leq \|v_t\|_{L^\infty(\R^n\times(0,T))} = M = Mh_\rho(x,t)$$
in the whole parabolic boundary of $\Omega_\rho$, and together with the fact that $(\p_t+L)h_\rho \geq 0$ in $\Omega_\rho$ we can conclude that $v_t \leq Mh_\rho$ in $\Omega_\rho$ by the comparison principle.

In particular, for every free boundary point $(x_0,t_0) \in \R^n\times[t_1,t_2]$, we have 
$$v_t(x_0,t_0+t) \leq Mt \quad \text{for} \quad t \in (0,\rho),$$
with uniform $M$ and $\rho$.

To conclude, observe that $v_t(x_0,t_0+t) \leq \rho^{-1}\|v_t\|_{L^\infty(\R^n\times(0,T))}t$ for all $t \geq \rho$, completing the proof.
\end{proof}

Now, using the previous estimate and the semiconvexity, we are ready to prove the global $C^{1,1}$ regularity of the solutions.

\begin{prop}\label{prop:C11}
Let $s \in (0,\frac{1}{2})$, and let $u$ be the solution of (\ref{eq:prev_problem}) with $L$ an operator satisfying (\ref{eq:operator}) and (\ref{eq:operator_elliptic}), and $\varphi \in C^{2,1}_c(\R^n)$. Let $0 < t_1 < t_2 < T$. Then, there exists $C > 0$ such that
$$\|D^2_xu\|_{L^\infty(\R^n\times[t_1,t_2])} + \|\p_t\nabla u\|_{L^\infty(\R^n\times[t_1,t_2])} + \|\p_{tt}u\|_{L^\infty(\R^n\times[t_1,t_2])} \leq C.$$
The constant $C$ depends only on $\varphi$, $t_1$, $t_2$, $T$, $s$, the dimension and the ellipticity constants.
\end{prop}

\begin{proof}
By Proposition \ref{prop:FB_es_mou}, there exists $\eta \in (0,1)$ such that $\eta|\nabla v| \leq v_t$ in $\R^n\times[\frac{t_1}{3},\frac{t_2+2T}{3}]$. Let $e$ be a vector in the $x$ directions with $|e| \leq 1$, and let $\nu = e_{n+1} + \eta e$. Thus, $\p_\nu v = (\p_t+\eta\p_e)v \geq 0$ in $\R^n\times[\frac{t_1}{3},\frac{t_2+2T}{3}]$.

Besides, for any given $(x,t) \in \R^n\times(0,T)$ and $r \in (0,2^{-1-\frac{1}{2s}}t)$, consider the cutoff $\psi \in C^\infty_c(Q_{2^{1+\frac{1}{2s}}r}(x,t))$ with $\psi \equiv 1$ in $Q_{2^{\frac{1}{2s}}r}(x,t)$. By Proposition \ref{prop:semiconvex}, since $|\nu| \leq \sqrt{2}$, $v_{\nu\nu} \geq -2\hat{C}$, and $\hat{C}$ does not depend on the choice of $\nu$. Then,

$$0 \leq \int_{Q_{2^{\frac{1}{2s}}r}(x,t)}\!v_{\nu\nu}+2\hat{C} \leq \int_{Q_{2^{1+\frac{1}{2s}}r}(x,t)}\!(v_{\nu\nu}+2\hat{C})\psi = \int_{Q_{2^{1+\frac{1}{2s}}r}(x,t)}\!v\psi_{\nu\nu} + 2\hat{C}\psi \leq C(r),$$
and then $\|v_{\nu\nu}\|_{L^1(Q_{2^{\frac{1}{2s}}r}(x,t))} \leq C(r) + 2\hat{C}|Q_{2^{\frac{1}{2s}}r}| =: C_1(r)$. Observe that this bound is independent of $(x,t)$ and $\nu$.

Then we define the auxiliary function
$$w := \frac{\p_\nu v(x+\eta he, t+h) - \p_\nu v(x,t)}{h} = \frac{1}{h}\int_0^h\p_{\nu\nu}v(x+\eta \zeta e,t+\zeta)\mathrm{d}\zeta.$$

Since $w$ is an average of $v_{\nu\nu}$, we can obtain a $L^1$ bound as well. Let $h \in (0,r)$. Then,
$$\|w\|_{L^1(Q_r(x,t))} \leq \frac{1}{h}\int_0^h\|v_{\nu\nu}\|_{L^1(Q_r(x+\eta \zeta e,t+h))}\mathrm{d}\zeta \leq \|v_{\nu\nu}\|_{L^1(Q_{2^{\frac{1}{2s}}r}(x,t))} = C_1(r).$$

This shows that $w \in L^1((t_3,t_4] \to L^1_s(\R^n))$ for any $t_3,t_4 \in (0,T-h]$. Let us compute it:

Let $r \in (0,2^{-1-\frac{1}{2s}}t_3)$ and $N = \lceil\frac{t_4-t_3}{2r}\rceil$. Then, we decompose the space in the following way:
\begin{align*}
&\|w\|_{L^1((t_3,t_4] \to L^1_s(\R^n))} \leq \sum\limits_{i = 0}^{N-1}\|w\|_{L^1((t_3 + 2ir,t_3 + 2(i+1)r] \to L^1_s(\R^n))} + \|w\|_{L^1((t_4-2r,t_4]\to L^1_s(\R^n))}\\
&= \sum\limits_{i = 0}^{N-1} \int_{t_3+2ir}^{t_3+2(i+1)r}\int_{\R^n}\frac{|w(x,t)|}{1+|x|^{n+2s}}\mathrm{d}x\mathrm{d}{t} + \int_{t_4-2ir}^{t_4}\int_{\R^n}\frac{|w(x,t)|}{1+|x|^{n+2s}}\mathrm{d}x\mathrm{d}{t}\\
&\leq \sum\limits_{i = 0}^{N-1} \sum\limits_{x \in \Z^n} \int_{t_3+2ir}^{t_3+2(i+1)r}\int_{B_r(rx/\sqrt{n})}\frac{|w(x,t)|}{1+|x|^{n+2s}}\mathrm{d}x\mathrm{d}{t}\\
&\quad + \sum\limits_{x \in \Z^n} \int_{t_4 - 2ir}^{t_4}\int_{B_r(rx/\sqrt{n})}\frac{|w(x,t)|}{1+|x|^{n+2s}}\mathrm{d}x\mathrm{d}{t}\\
&= \sum\limits_{i = 0}^{N-1} \sum\limits_{x \in \Z^n} \int_{Q_r(rx/\sqrt{n},t_3+(2i+1)r)}\frac{|w(x,t)|}{1+|x|^{n+2s}}\mathrm{d}x\mathrm{d}{t} +  \sum\limits_{x \in \Z^n} \int_{Q_r(rx/\sqrt{n},t_4-r)}\frac{|w(x,t)|}{1+|x|^{n+2s}}\mathrm{d}x\mathrm{d}{t}\\
&\leq N\sum\limits_{x\in \Z^n}\frac{C_1(r)}{1 + (|rx/\sqrt{n}|-r)_+^{n+2s}} =: NC_2(r).
\end{align*}

Moreover, let $\tau_yw$ be the translation of $w$ by the vector $y \in \R^n$. Analogously, we can deduce that
$$\|\tau_yw\|_{L^1((t_3,t_4]\to L^1_s(\R^n))} \leq NC_2(r),$$
independently of $y$.

Now, recall that $v$ is a solution of $(\p_t + L)v = -L\varphi$ in the set $\{v > 0\}$. Furthermore, if $v > 0$ at $(x,t) \in \R^n\times[\frac{t_1}{3},\frac{t_2+2T}{3}]$, since $\p_\nu v \geq 0$, $v(x+\eta he,t+h) > 0$ also holds (provided that $t + h \leq \frac{t_2+2T}{3}$), and it follows that the translated function is also a solution. Hence,
$$\p_tw+Lw = \eta\frac{\p_e L\varphi(x) - \p_e L\varphi(x + \eta h e)}{h} \quad \text{in} \quad \{v > 0\}\cap\left(\R^n\times\left[\frac{t_1}{3},\frac{t_2+2T}{3} - h\right]\right),$$
and then $|\p_tw + Lw| \leq C\|L\varphi\|_{C^{1,1}(\R^n)} \leq C\|\varphi\|_{C^{2,1}(\R^n)}$ in 
$$\{v > 0\}\cap\left(\R^n\times\left[\frac{t_1}{3},\frac{t_2+2T}{3} - h\right]\right) \subset \{v > 0\}\cap\left(\R^n\times\left[\frac{2t_1}{3},\frac{2t_2+T}{3}\right]\right),$$
provided that $h$ is small enough.

Moreover, if $(x_1,t_1) \in \{v = 0\}\cap(\R^n\times[\frac{2t_1}{3},\frac{2t_2+T}{3}])$, then $\p_\nu v(x_1,t_1) = 0$, and using Proposition \ref{prop:anti-hopf} and taking $h$ small enough, it follows that
\begin{align*}
    w(x_1,t_1) &= \frac{\p_\nu v(x_1+\eta he,t_1+h)}{h} \leq \frac{2v_t(x_1+\eta he,t_1+h)}{h}\\
    &\leq \frac{2M(t_1+h-\Gamma(x_1+\eta he))_+}{h} \leq \frac{2M(h+C_0\eta h|e|+t_1-\Gamma(x_1))}{h}\\
    &\leq 4M.
\end{align*}
Therefore, $\tilde{w} = \max\{w,4M\}$ is a subsolution for
$$\p_t\tilde{w}+L\tilde{w}\leq C\|\varphi\|_{C^{2,1}(\R^n)} \quad \text{in} \quad \R^n\times\left[\frac{2t_1}{3},\frac{2t_2+T}{3}\right],$$
and we can apply Lemma \ref{lem:CD16} to $\tau_y\tilde{w}$ obtain
$$\sup\limits_{B_1\times[t_1,t_2]} \tau_y\tilde{w} \leq C\left(\|\tau_y\tilde{w}\|_{L^1\left(\left(\frac{2t_1}{3},\frac{2t_2+T}{3}\right]\to L^1_s(\R^n)\right)} + \|\varphi\|_{C^{2,1}(\R^n)}\right),$$
with $C$ depending only on $t_1$, $t_2$, $T$, the dimension, the ellipticity constants and $s$. Then, since the bound is uniform on $y$, it follows from the definition of $w$ that
$$\sup\limits_{\R^n\times[t_1,t_2]} w \leq C(NC_2(r) + 2M + \|\varphi\|_{C^{2,1}(\R^n)}) =: C_0.$$

Since $C_0$ does not depend on $\nu$ or $h$, combining this with Proposition \ref{prop:semiconvex}, it follows that $\|v_{\nu\nu}\|_{L^\infty(\R^n\times[t_1,t_2])} \leq C_* = \max\{C_0,2\hat{C}\}$ for all $\nu = e_{n+1} + \eta e$ with $e$ in the $x$ direction and $|e| < 1$.

Now, let $e = \lambda\hat{e}$ with $\hat{e}$ a unit vector. Then,
$$D^2_{e_{n+1}+\eta e}v = v_{tt} +\eta(v_{te} + v_{et}) + \eta^2 v_{ee} = v_{tt} + \eta\lambda(v_{t\hat{e}}+v_{\hat{e}t}) + \eta^2\lambda^2v_{\hat{e}\hat{e}}.$$

Since this expression is bounded by $C_*$ for all values of $\hat{e}$ and $\lambda \in (-1,1)$, we can evaluate at $\lambda = 0,\frac{1}{2},-\frac{1}{2}$ to get:
\begin{align*}
    &|v_{tt}| \leq C_*\\
    &\left|v_{tt} + \frac{1}{2}\eta(v_{t\hat{e}} + v_{\hat{e}t}) + \frac{1}{4}\eta^2v_{\hat{e}\hat{e}}\right| \leq C_*\\
    &\left|v_{tt} - \frac{1}{2}\eta(v_{t\hat{e}} + v_{\hat{e}t}) + \frac{1}{4}\eta^2v_{\hat{e}\hat{e}}\right| \leq C_*,
\end{align*}
and then it is easy to check that $|v_{\hat{e}\hat{e}}| + |v_{t\hat{e}} + v_{\hat{e}t}| \leq C(\eta)C_*$.

Hence, for any $e \in \mathbb{S}^n$ (all unit vectors in $x,t$), $|v_{ee}| \leq C'(\eta)C_*$. Then, given two points $(x_1,t_1)$ and $(x_2,t_2)$ in $\R^n\times[t_1,t_2]$,
$$|v(x_1,t_1) - \nabla_{x,t}v(x_1,t_1)\cdot(x_2-x_1,t_2-t_1) - v(x_2,t_2)| \leq C'(\eta)C_*\|(x_1-x_2,t_1-t_2)\|^2.$$
This means that $v \in C^{1,1}(\R^n\times[t_1,t_2])$, and $u = v + \varphi$ as well.
\end{proof}

We can now give the:
\begin{proof}[Proof of Theorem \ref{thm:1.1}]
The global Lipschitz regularity follows from Proposition \ref{prop:global_lip}. The $C^{1,1}$ regularity follows from Proposition \ref{prop:C11}.
\end{proof}

\section{Regularity of the free boundaries}\label{sect:FB}
In this section we use the regularity of the solutions established before to deduce the regularity of the free boundaries. Here again, we will use crucially the fact that $s < \frac{1}{2}$. We first take advantage of the different orders of derivation in the equation (\ref{eq:prev_problem}) to obtain further regularity in $t$.

\begin{lem}\label{lem:C2-2s}
Let $s \in (0,\frac{1}{2})$, let $u$ be the solution of (\ref{eq:prev_problem}) with $L$ an operator satisfying (\ref{eq:operator}) and (\ref{eq:operator_elliptic}), and $\varphi \in C^{2,1}_c(\R^n)$. Let $v = u - \varphi$, and let $0 < t_1 < t_2 < T$. Then, there exists $C > 0$ such that
$$\|v_{tt}\|_{C^{\alpha}((\R^n\times[t_1,t_2]) \cap \{v > 0\})} + \sum\limits_{i = 0}^n\|v_{ti}\|_{C^{\alpha}((\R^n\times[t_1,t_2]) \cap \{v > 0\})} \leq C.$$
where $\alpha = 1 - 2s > 0$.
\end{lem}

\begin{proof}
Let $\nu \in \mathbb{S}^n$ be any unit vector in $x$ and $t$, and let $w = \p_\nu u$. Then, by Proposition \ref{prop:C11}, $\|w\|_{C^{0,1}(\R^n\times[t_1,t_2])} \leq C$. Moreover, by the same arguments as in the proof of Proposition \ref{prop:t_C1alpha}, we deduce $\|Lw\|_{C^\alpha(\R^n\times[t_1,t_2])} \leq C$.

Then, since $v_t = u_t = -Lu$ in $\{v > 0\}$, differentiating the equation with respect to $\nu$ it follows that $w_t = -Lw$ in $\{v > 0\}$, and therefore $\|v_{t\nu}\|_{C^\alpha(\R^n\times[t_1,t_2])} \leq C$.
\end{proof}

We next show that the free boundary is $C^{1,\alpha}$.
\begin{thm}\label{thm:FB_C1alpha}
Let $s \in (0,\frac{1}{2})$, and let $u$ be the solution of (\ref{eq:prev_problem}) with $L$ an operator satisfying (\ref{eq:operator}) and (\ref{eq:operator_elliptic}), and $\varphi \in C^{2,1}_c(\R^n)$. Let $0 < t_1 < t_2 < T$.

Then, the free boundary is a $C^{1,\alpha}$ graph in the $t$ direction in $\R^n\times[t_1,t_2]$, i.e.
$$\p\{u > \varphi\}\cap(\R^n\times(t_1,t_2)) = \{t = \Gamma(x)\},$$
with $\Gamma \in C^{1,\alpha}$ and $\alpha = 1 - 2s > 0$.
\end{thm}

\begin{proof}
We already know that the free boundary is a Lipschitz graph by Proposition \ref{prop:FB_es_mou}. Then, let $\alpha = 1 - 2s$. By Lemma \ref{lem:C2-2s},
$$\|v_{tt}\|_{C^{\alpha}((\R^n\times[t_1,t_2]) \cap \{v > 0\})} + \sum\limits_{i = 0}^n\|v_{ti}\|_{C^{\alpha}((\R^n\times[t_1,t_2]) \cap \{v > 0\})} \leq C.$$

Then, $v_{tt} > 0$ at the free boundary by Proposition \ref{prop:hopf}, and by continuity $v_{tt} \geq c_0$ in  $E = \{t \in [\Gamma(x), \Gamma(x) + \delta]\}\cap[t_1,t_2]$ for some small $\delta > 0$. Thus,
$$\left\|\frac{v_{ti}}{v_{tt}}\right\|_{C^{\alpha}(E)} \leq C.$$

Finally, notice that the free boundary can be seen as the zero level surface of $v_t$. The normal vector to the level surfaces of $v_t$ is given by the formula
$$\nu = \frac{\nabla_{x,t} v_t}{|\nabla_{x,t} v_t|} = \frac{(\p_{t1}v/\p_{tt}v,\ldots,\p_{tn}v/\p_{tt}v,1)}{\sqrt{1+\sum\limits_{j = 1}^n(\p_{tj}v/\p_{tt}v)^2}},$$
and therefore $\nu \in C^{\alpha}(E)$ uniformly, thus $\{v_t = 0\}$ is a $C^{1,\alpha}$ manifold, as desired.
\end{proof}

Once we know that the free boundary is a $C^{1,\alpha}$ graph, we can provide an expansion for the solution.
\begin{cor}\label{cor:expansion}
Let $s \in (0,\frac{1}{2})$, and let $u$ be the solution of (\ref{eq:prev_problem}) with $L$ an operator satisfying (\ref{eq:operator}) and (\ref{eq:operator_elliptic}), and $\varphi \in C^{2,1}_c(\R^n)$. Let $(x_0,t_0) \in \p\{u > \varphi\}$ be a free boundary point. Then,
$$u_t(x_0+x,t_0+t) = c_0(t-a\cdot x)_+ + O(t^{1+\alpha}+|x|^{1+\alpha})$$
and
$$(u - \varphi)(x_0+x,t_0+t) = \frac{c_0}{2}(t-a\cdot x)_+^2 + O(t^{2+\alpha}+|x|^{2+\alpha}),$$
with $\alpha = 1 - 2s > 0$, $c_0 = u_{tt}(x_0, t_0) > 0$ and $a = \nabla\Gamma(x_0)$.
\end{cor}

\begin{proof}
We will use strongly that $\Gamma \in C^{1,\alpha}$ by Theorem \ref{thm:FB_C1alpha}, and that $u_t \in C^{1,\alpha}(\overline{\{u > \varphi\}})$ by Lemma \ref{lem:C2-2s}.

We distinguish two cases. If $(x_0+x,t_0+t) \in \{u = \varphi\}$, $t_0+t \leq \Gamma(x_0+x)$, then expanding $\Gamma(x_0+x) = t_0 + \nabla\Gamma(x_0)\cdot x + O(|x|^{1+\alpha})$ we obtain
$$t - \nabla\Gamma(x_0)\cdot x \leq O(|x|^{1+\alpha}),$$
and therefore
$$(t - \nabla\Gamma(x_0)\cdot x)_+^2 \leq O(|x|^{2+2\alpha}) \leq O(|x|^{2+\alpha}),$$
and since $(u-\varphi)(x_0+x,t_0+t) = u_t(x_0+x,t_0+t) = 0$ this is exactly what we needed.

On the other hand, outside of the contact set,
\begin{align*}
    u_t(x_0+x,t_0+t) &= \int_{\Gamma(x_0+x)}^{t_0+t}u_{tt}(x_0+x,\tau)\mathrm{d}\tau\\
    &= (t_0+t-\Gamma(x_0+x))(u_{tt}(x_0,t_0)+O(t^\alpha+|x|^\alpha))\\
    &= (t - \nabla\Gamma(x_0)\cdot x)_+u_{tt}(x_0,t_0) + O(t^{1+\alpha} + |x|^{1+\alpha}),
\end{align*}
where in the last equality we expanded $\Gamma(x_0+x)$ as before, and if $t - \nabla\Gamma(x_0)\cdot x \leq 0$, the whole term is $O(t^{1+\alpha}+|x|^{1+\alpha})$ and can be absorbed in the error term because $t_0 + t - \Gamma(x_0+x) \geq 0$.

Then, we can repeat the procedure and integrate $u_t$, knowing already its expansion, and the conclusion follows from an analogous computation.
\end{proof}

We can now give the:
\begin{proof}[Proof of Theorem \ref{thm:1.2}]
The first part is Theorem \ref{thm:FB_C1alpha}, the second part is Corollary \ref{cor:expansion}.
\end{proof}

\subsection{Regular and singular points}

\begin{defn}\label{defn:regular_singular_pts}
Let $s \in (0,\frac{1}{2})$, and let $u$ be the solution of (\ref{eq:prev_problem}) with $L$ an operator satisfying (\ref{eq:operator}) and (\ref{eq:operator_elliptic}), and $\varphi \in C^{2,1}_c(\R^n)$. Let $(x_0,t_0) \in \p\{u > \varphi\}$ be a free boundary point. Then,
\begin{itemize}
    \item We say $(x_0,t_0)$ is a \textit{regular} free boundary point if there exists $c_0 > 0$ such that for all small $r > 0$,
    $$\|u(\cdot,t_0) - \varphi\|_{L^\infty(B_r(x_0))} \geq c_0r^2.$$
    \item We say $(x_0,t_0)$ is a \textit{singular} free boundary point if it is not regular.
\end{itemize}
\end{defn}

One important first observation is the following.

\begin{prop}\label{prop:regular_open}
Let $s \in (0,\frac{1}{2})$, and let $u$ be the solution of (\ref{eq:prev_problem}) with $L$ an operator satisfying (\ref{eq:operator}) and (\ref{eq:operator_elliptic}), and $\varphi \in C^{2,1}_c(\R^n)$. Then, if $(x_0,t_0)$ is any free boundary point, the following are equivalent:
\begin{enumerate}
    \item $(x_0,t_0)$ is a regular free boundary point.
    \item If $\nu_0$ is the normal vector to the free boundary at $(x_0,t_0)$, $\nu_0 \neq e_{n+1}$.
    \item $\nabla u_t(x_0,t_0) \neq 0$.
\end{enumerate}

Moreover, the set of regular free boundary points is an open subset of $\p\{u > \varphi\}$.
\end{prop}

\begin{proof}
(ii) $\Leftrightarrow$ (iii):

It follows directly from
$$\nu_0 = \frac{(\nabla u_t(x_0,t_0), u_{tt}(x_0,t_0))}{\sqrt{1 + |\nabla u_t (x_0,t_0)|^2/u_{tt}(x_0,t_0)^2}}$$
and the fact that $u_{tt}(x_0,t_0) > 0$.

(i) $\Leftrightarrow$ (ii):

We will distinguish the cases $\nu_0 = e_{n+1}$ and $\nu_0 \neq e_{n+1}$. If $\nu_0 = e_{n+1}$, let $\{t = \Gamma(x)\}$ be the free boundary. Then, $\Gamma \in C^{1,\alpha}$ and $\nabla\Gamma(x_0) = 0$ because $\nu_0 = e_{n+1}$. Then,
$$\Gamma(x_0+x) \geq t_0 - C|x|^{1+\alpha},$$
and therefore
\begin{align*}
    (u - \varphi)(x_0+x,t_0) &\leq \int_{\Gamma(x_0+x)}^{t_0}\int_{\Gamma(x_0+x)}^\tau u_{tt}\mathrm{d}\tau'\mathrm{d}\tau\\
    &\leq \frac{(t_0 - \Gamma(x_0+x))^2}{2}\|u_{tt}\|_{L^\infty(\R^n\times[\Gamma(x_0+x),t_0])}\\
    &\leq C|x|^{2+2\alpha},
\end{align*}
contradicting the assumption that $(x_0,t_0)$ is a regular point.

On the other hand, if $\nu_0 = \alpha e_{n+1} + \beta e$, with $e$ a unit vector in the $x$ directions and $\beta > 0$, we can also approximate $\Gamma$ as
$$\Gamma(x_0+x) \leq t_0 - \frac{\beta}{\alpha}(x\cdot e) + C|x|^{1+\alpha}.$$
Notice that $\alpha \neq 0$ because $u_{tt} > 0$ on the free boundary as a consequence of Proposition \ref{prop:hopf}. We also need to use that, for some small $\delta > 0$, $u_{tt} \geq c_\delta > 0$ in the set $E_\delta = \{t \in [\Gamma(x),\Gamma(x)+\delta]\}\cap[t_0-\delta,t_0+\delta]$, by the same argument as in the proof of Theorem \ref{thm:FB_C1alpha}.

Then, if $r$ is small,
$$\|u(\cdot,t_0) - \varphi\|_{L^\infty(B_r(x_0))} \geq u(x_0 + \frac{r}{2}e,t_0) - \varphi(x_0 + \frac{r}{2}e) \geq \frac{1}{2}\left(\frac{\beta r}{2\alpha} - Cr^{1+\alpha}\right)^2c_\delta \geq c_0r^2.$$

For the last part, first notice that $\nabla u_t$ is a continuous function in $\overline{\{u > \varphi\}}$ because $u_t \in C^{1,\alpha}(\{u > \varphi\})$ by Lemma \ref{lem:C2-2s}. As a consequence, the set of regular points, $\{\nabla u_t \neq 0\}\cap\p\{u > \varphi\}$, is a relatively open set.
\end{proof}

In a neighbourhood of a regular free boundary point, the free boundary is also $C^{1,\alpha}$ in space:

\begin{prop}\label{prop:FB_regular_pts}
Let $s \in (0,\frac{1}{2})$, and let $u$ be the solution of (\ref{eq:prev_problem}) with $L$ an operator satisfying (\ref{eq:operator}) and (\ref{eq:operator_elliptic}), and $\varphi \in C^{2,1}_c(\R^n)$. Let $(x_0,t_0)$ be any regular free boundary point.

Then, there exists an open neighbourhood $x_0 \in U \subset \R^n\times(0,T)$ such that the free boundary is a $C^{1,\alpha}$ graph in the $x$ direction, i.e., there exists $i \in \{0,\ldots,n\}$ such that
$$\p\{u > \varphi\}\cap U = \{x_i = F_i(x_1,\ldots,x_{i-1},x_{i+1},\ldots,x_n,t)\},$$
with $F_i \in C^{1,\alpha}$ and $\alpha = 1 - 2s > 0$.
\end{prop}

\begin{proof}
First, by Theorem \ref{thm:FB_C1alpha}, the free boundary can be represented as\linebreak $\p\{u > \varphi\} = \{t = \Gamma(x)\}$ in a neighbourhood of $(x_0,t_0)$, with $\Gamma \in C^{1,\alpha}$. Moreover, since $(x_0,t_0)$ is regular, by Proposition \ref{prop:regular_open}, the normal vector to the free boundary $\nu_{(x_0,t_0)} \neq e_{n+1}$, and thus $\nabla\Gamma(x_0) \neq 0$, and in particular $\p_{x_i}\Gamma(x_0) \neq 0$.

Therefore, by the implicit function theorem, $\{u > \varphi\}\cap\{t = t_0\}$ is locally a $C^{1,\alpha}$ graph of the form $(x_1\ldots,x_{i-1},x_{i+1},\ldots,x_n,t) \mapsto x_i$.
\end{proof}

On the other hand, in the time slice of a singular point, the free boundary could be very complicated. Nevertheless, we can prove that singular points are scarce. To do so, we will use the following lemma from geometric measure theory.

\begin{lem}[\protect{\cite{FRS20}}]\label{lem:haussdorf_dim}
Consider the family $\{E_t\}_{t \in (0,T)}$ with $E_t \subset \R^n$, and let us denote $E := \bigcup\limits_{t \in (0,T)} E_t$.

Let $1 \leq \gamma \leq \beta \leq n$, and assume that the following holds:
\begin{itemize}
    \item $\operatorname{dim}_\mathcal{H} E_t \leq \beta$,
    \item for all $\varepsilon > 0$, $t_0 \in (0,T)$ and $x_0 \in E_{t_0}$, there exists $\rho > 0$ such that
    $$B_r(x_0) \cap E_t = \emptyset,$$
    for all $r \in (0,\rho)$ and $t > t_0 + r^{\gamma - \varepsilon}$.
\end{itemize}

Then, $\operatorname{dim}_{\mathcal{H}} E_t \leq \beta - \gamma$, for $\mathcal{H}^1$-a.e. $t \in (0,T)$.
\end{lem}

Using the global $C^{1,\alpha}$ regularity of the free boundary, and noticing that the normal vector is $e_{n+1}$ at singular points, we can prove the following dimension bound.

\begin{prop}\label{prop:singular_dim}
Let $s \in (0,\frac{1}{2})$, and let $u$ be the solution of (\ref{eq:prev_problem}) with $L$ an operator satisfying (\ref{eq:operator}) and (\ref{eq:operator_elliptic}), and $\varphi \in C^{2,1}_c(\R^n)$. Let $\Sigma \subset \p\{u > \varphi\}$ be the set of singular free boundary points, and let $\Sigma_t = \{(x,t') \in \Sigma : t' = t\}$ be the time slices of the singular set.

Then, 
$$\operatorname{dim}_{\mathcal{H}} \Sigma_t \leq n - 1 - \alpha, \quad \text{for almost every} \quad t \in (0,T),$$
with $\alpha = 1 - 2s > 0$. In particular, $\mathcal{H}^{n-1}(\Sigma_t) = 0$ for almost every $t \in (0,T)$.
\end{prop}

\begin{proof}
We just need to check the hypotheses of Lemma \ref{lem:haussdorf_dim}, with $\beta = n$ and $\gamma = 1+\alpha$. The first condition is obvious, because since $\Sigma_t \subset \R^n\times\{t\}$, $\operatorname{dim}_{\mathcal{H}}\Sigma_t \leq n$.

For the second condition, we use the $C^{1,\alpha}$ regularity of the free boundary. Let $x_0 \in E_{t_0}$. This means that $(x_0,t_0)$ is a singular free boundary point. In particular, since $v_t(x_0,\Gamma(x_0)) = 0$ and $v_{tt}(x_0,t_0) \neq 0$,
$$\nabla \Gamma (x_0) = -\frac{\nabla v_t(x_0,t_0)}{v_{tt}(x_0,t_0)} = 0.$$

Now, $\Gamma \in C^{1,\alpha}$. Therefore, $\Gamma(x) \leq t_0 + C|x - x_0|^{1+\alpha}$ for all $x \in B_\rho(x_0)$ for some $\rho > 0$.

Finally, for any $\varepsilon > 0$, there exists $\rho(\varepsilon)$ such that for all $r \in (0,\rho(\varepsilon))$,
$$\Gamma(x) \leq t_0 + Cr^{1+\alpha} < t_0 + r^{1+\alpha-\varepsilon},$$
and thus $B_r(x_0)\cap\Sigma_t = \emptyset$ for all $t > t_0 + r^{1+\alpha-\varepsilon}$, completing the proof.
\end{proof}

We finally give the:
\begin{proof}[Proof of Theorem \ref{thm:1.3}]
The first part follows from Proposition \ref{prop:regular_open}, the second is Proposition \ref{prop:FB_regular_pts} and the last is Proposition \ref{prop:singular_dim}.
\end{proof}

\appendix
\section{Some tools for nonlocal parabolic equations}\label{sect:appendix_heat}

We start recalling the following estimates on the fundamental solution to the nonlocal heat equation, see \cite{CK03}.
\begin{thm}[\protect{\cite{CK03}}]\label{thm:kernel_estimates}
Let $L$ be an operator satisfying (\ref{eq:operator}) and (\ref{eq:operator_elliptic}), and let $w \in L^\infty(\R^n\times(0,T))$ be the solution of
$$\begin{cases}
(\p_t+L)w = 0\quad \text{in} \quad \R^n\times(0,T)\\
w = w_0 \quad \text{on} \quad \{t = 0\}.
\end{cases}$$
Then,
$$w(x,t) = p_t * w_0,$$
and $p_t$ is nonnegative, $\|p_t(\cdot,t)\|_{L^1(\R^n)} = 1$ for all $t \in (0,T)$,
$$(\p_t+L)p_t = 0 \quad \text{in} \quad \R^n\times(0,T),$$
and
$$c_1\min\{t^{-\frac{n}{2s}},t|x|^{-n-2s}\} \leq p_t(x) \leq c_2\min\{t^{-\frac{n}{2s}},t|x|^{-n-2s}\},$$
for some $0 < c_1 < c_2$ depending only on $T$, the dimension, $s$ and the ellipticity constants.
\end{thm}

It is worth noticing that $p_t$ is an approximation to the identity, in the following sense.
\begin{cor}\label{cor:heat_kernel_aproxid}
Let $f \in L^\infty(\R^n)$ be uniformly continuous, and define $f_t = p_t * f$ for all $t > 0$, with $p_t$ the fundamental solution introduced in Theorem \ref{thm:kernel_estimates}. Then,
$$\|f_t\|_{L^\infty(\R^n)} \leq \|f\|_{L^\infty(\R^n)}$$
and
$$\|f_t - f\|_{L^\infty(\R^n)} \rightarrow 0 \quad \text{as} \quad t \rightarrow 0.$$
\end{cor}

\begin{proof}
Since $p_t \geq 0$ and $\|p_t(\cdot, t)\|_{L^1(\R^n)} = 1$, the trivial bound of the convolution suffices to obtain the first inequality.

For the second inequality, for any $\varepsilon > 0$ and any $x \in \R^n$,
\begin{align*}
    |f_t(x) - f(x)| &= \left|\int_{\R^n} p_t(y)(f(x-y) - f(x))\mathrm{d}y\right|\\
    &\leq \int_{B_\delta} p_t(y)|f(x-y) - f(x)|\mathrm{d}y + \int_{B_\delta^c} p_t(y)|f(x-y) - f(x)|\mathrm{d}y\\
    &\leq \varepsilon\int_{B_\delta}p_t + 2\|f\|_{L^\infty(\R^n)}\int_{B_\delta^c}p_t \leq \varepsilon + 2c_2\delta^{-2s}\|f\|_{L^\infty(\R^n)}t < 2\varepsilon,
\end{align*}
as we can choose $\delta$ sufficiently small to ensure $|f(x-y)-f(x)| < \varepsilon$ inside $B_\delta$ by uniform continuity, and then use Theorem \ref{thm:kernel_estimates} and make $t$ tend to zero.
\end{proof}

We will also use the following $L^1$ to $L^\infty$ bound for subsolutions.

\begin{lem}\label{lem:CD16}
Let $L$ be an operator satisfying (\ref{eq:operator}) and (\ref{eq:operator_elliptic}), and let $w \in L^\infty(\R^n\times(-1,0))$ be a subsolution of
$$(\p_t+L)w \leq C_0 \quad \text{in} \quad \R^n\times(-1,0).$$
Then, 
$$\sup\limits_{B_1\times[-1+\delta,0)}w \leq C\left(\int_{-1}^0\int_{\R^n}\frac{|w(x,t)|}{1+|x|^{n+2s}}\mathrm{d}x\mathrm{d}t+ C_0\right),$$
where $C$ depends only on $\delta > 0$, $s$, the dimension and the ellipticity constants.
\end{lem}

\begin{proof}
Since $w - C_0(t+1) \geq w - C_0$ and $(\p_t+L)(w - C_0(t+1)) \leq 0$, we can assume without loss of generality that $C_0 = 0$.

Then, since $w$ is a subsolution for the nonlocal heat equation, the following holds for any $-1 < t_0 < t < 0$:
$$w(x,t) \leq \int_{\R^n}p_{t-t_0}(x-y)w(y,t_0)\mathrm{d}y,$$
where $p_t(x)$ is the heat kernel associated to the operator $L$ (see Theorem \ref{thm:kernel_estimates}). Then, given $\delta > 0$, $x \in B_1$ and $t \in [-1+\delta,0)$ we can integrate the relation in time to obtain the following:
\begin{align*}
    w(x,t) &\leq \int_{t-\delta}^{t-\frac{\delta}{2}}\int_{\R^n}p_{t-\zeta}(x-y)|w(y,\zeta)|\mathrm{d}y\mathrm{d}\zeta\\
    &\leq \int_{t-\delta}^{t-\frac{\delta}{2}}\int_{\R^n} C\min\{(t-\zeta)^{-\frac{n}{2s}},(t-\zeta)|x-y|^{-n-2s}\}|w(y,\zeta)|\mathrm{d}y\mathrm{d}\zeta\\
    &\leq C\int_{t-\delta}^{t-\frac{\delta}{2}}\int_{\R^n}\frac{2}{(t-\zeta)^{\frac{n}{2s}}+(t-\zeta)^{-1}|x-y|^{n+2s}}|w(y,\zeta)|\mathrm{d}y\mathrm{d}\zeta\\
    &\leq C\int_{t-\delta}^{t-\frac{\delta}{2}}\int_{\R^n}\frac{2}{(\frac{\delta}{2})^{\frac{n}{2s}}+\delta^{-1}|x-y|^{n+2s}}|w(y,\zeta)|\mathrm{d}y\mathrm{d}\zeta\\
    &\leq C\int_{t-\delta}^{t-\frac{\delta}{2}}\int_{\R^n} \frac{1}{1+|y|^{n+2s}}|w(y,\zeta)|\mathrm{d}y\mathrm{d}\zeta,
\end{align*}
and $C$ depends on $\delta$, and universal constants ($n$, $s$, $\lambda$ and $\Lambda$).
\end{proof}

For the interior regularity, we will need an analogue of \cite[Corollary 3.4]{FR17}.

\begin{prop}\label{prop:interior_alphabeta}
Let $L$ be an operator satisfying (\ref{eq:operator}) and (\ref{eq:operator_elliptic}). Let $u \in L^\infty(\R^n\times(-1,0))$ be a viscosity solution of $u_t + Lu = f$ in $B_1\times(-1,0)$. Assume additionally that
\begin{align*}
C_0 = &\sup\limits_{t\in(-1,0)}\|u(\cdot,t)\|_{C^{\alpha}(\R^n)} + \sup\limits_{x \in \R^n}\|u(x,\cdot)\|_{C^{\beta}((-1,0))}\\
     + &\sup\limits_{t\in(-1,0)}\|f(\cdot,t)\|_{C^{\alpha}(B_1)} + \sup\limits_{x \in B_1}\|f(x,\cdot)\|_{C^{\beta}((-1,0))} < \infty,
\end{align*}
for some $\alpha, \beta \geq 0$ (with the $L^\infty$ norm if $\alpha$ or $\beta$ are $0$).

Then, for all $\varepsilon > 0$, $u \in C^{\alpha+2s-\varepsilon}_xC^{\beta+1-\varepsilon}_t(\overline{B_{1/2}}\times[-\frac{1}{2},0])$, and
$$\sup\limits_{t\in[-\frac{1}{2},0]}\|u(\cdot,t)\|_{C^{\alpha+2s-\varepsilon}(\overline{B_{1/2}})} + \sup\limits_{x \in \overline{B_{1/2}}}\|u(x,\cdot)\|_{C^{\beta+1-\varepsilon}([-\frac{1}{2},0])} \leq CC_0,$$
where $C$ only depends on the dimension, $s$, $\varepsilon$, and the ellipticity constants.
\end{prop}

\begin{proof}
The proof is the same as the proof of \cite[Corollary 3.4]{FR17}, but using \cite[Theorem 2.2]{Ser15} instead of \cite[Theorem 1.3]{FR17}.
\end{proof}

Combining the heat kernel estimates with the interior regularity result, we obtain the following bound.

\begin{cor}\label{cor:grad_kernel}
Let $L$ be an operator satisfying (\ref{eq:operator}) and (\ref{eq:operator_elliptic}), and let $p_t$ as introduced in Theorem \ref{thm:kernel_estimates}. Then, for all $r_0 > 0$,
$$\|\nabla p_t\|_{L^\infty(B_{r_0}^c\times(0,T))} \leq C,$$
where $C$ depends only on $r_0$, $T$, the dimension, $s$ and the ellipticity constants.
\end{cor}

\begin{proof}
Assume after a scaling that $r_0 = 1$. Iterating proposition \ref{prop:interior_alphabeta}, we obtain that
$$\sup_{t \in [-2^{-k},0]}\|p_t(\cdot,t)\|_{C^1(\overline{B_{2^{-k}}})} \leq C\|p_t\|_{L^\infty(B_1\times(-1,0))},$$
for some big enough $k$ depending only on $s$. After a scaling and a covering argument, for all $x\in B_1^c$ it holds
$$\begin{array}{rclll}
\|\nabla p_t\|_{L^\infty(B_{\delta/2}(x)\times[t-\frac{\delta^{2s}}{2},t])} &\leq &C_\delta\|p_t\|_{L^\infty(B_\delta(x)\times(t-\delta^{2s},t))}, &\text{for all} &t\in(\delta^{2s},T),\\
\|\nabla p_t\|_{L^\infty(B_{t^{\frac{1}{2s}}/2}(x)\times[\frac{t}{2},t])} &\leq &C_0t^{-1}\|p_t\|_{L^\infty(B_{t^{\frac{1}{2s}}}(x)\times(0,t))}, &\text{for all} &t\in(0,T),
\end{array}$$
where we leave $\delta > 0$ to be chosen later.

Then, using Theorem \ref{thm:kernel_estimates}, substituting the estimate $|p_t(x)| \leq c_2t|x|^{-n-2s}$,
$$\|\nabla p_t\|_{L^\infty(B_{\delta/2}(x)\times[t-\frac{\delta^{2s}}{2},t])} \leq C_\delta c_2t(1-\delta)^{-n-2s}, \quad \text{for all} \ t\in(\delta^{2s},T),$$
and
$$\|\nabla p_t\|_{L^\infty(B_{t^{\frac{1}{2s}}/2}(x)\times[\frac{t}{2},t])} \leq C_0c_2(1-t^{\frac{1}{2s}})_+^{-n-2s}, \quad \text{for all} \ t\in(0,T).$$

Finally, choosing $\delta = \frac{1}{4}$, for all $x \in B_1^c$ and $t \geq 4^{-2s}$, 
$$|\nabla p_t(x,t)| \leq C_{1/4}c_2t(3/4)^{-n-2s} \leq C_{1/4}c_2T(3/4)^{-n-2s},$$
and for all $x \in B_1^c$ and $t \in (0,4^{-2s})$, 
$$|\nabla p_t(x,t)| \leq C_0c_2(3/4)^{-n-2s},$$
as we wanted to prove.
\end{proof}

We will also make use of the following estimate for the nonlocal heat equation.

\begin{prop}\label{prop:barrier}
Let $L$ be an operator satisfying (\ref{eq:operator}) and (\ref{eq:operator_elliptic}). Then, there exists $\delta > 0$ such that the following holds. If $b \in L^\infty$ is continuous and satisfies
$$\left\{\begin{array}{rclll}
|(\p_t+L)b| & \leq & \delta \max\{|x|, 1\}^{-n-2s} & \text{in} & \R^n\times(0,1)\\
b & = & b_0 & \text{on} & \{t = 0\},
\end{array}\right.$$
where $b_0 \geq 0$, $\operatorname{supp} b_0 \subset B_1$ and $\|b_0\|_{L^1(B_1)} = 1$, the following estimate holds:
$$c_1t|x|^{-n-2s} \leq b(x,t) \leq c_2t|x|^{-n-2s} \quad \text{for all} \quad (x,t) \in B_2^c\times(0,1)$$
The constants $\delta$, $c_1$ and $c_2$ are positive and depend only on the dimension, $s$ and the ellipticity constants.
\end{prop}

\begin{proof}
We will use Duhamel's formula with the fundamental solution, together with Theorem \ref{thm:kernel_estimates}. Let us take $\delta = 0$ first and then we will show that the perturbation introduced by the right hand side can be absorbed by the constants.

If $|x| > 2$ and $t < 1$, $p_t(x) \asymp t|x|^{-n-2s}$. Thus, if $|x| \geq 2$, for all $y \in B_1$, $|x-y| \asymp |x|$, and then
\begin{align*}
    b(x,t) &= \int_{\R^n}p_t(x-y)b_0(y)\mathrm{d}y = \int_{B_1}p_t(x-y)b_0(y)\mathrm{d}y\\
    &\asymp \int_{B_1}t|x|^{-n-2s}b_0(y)\mathrm{d}y = t|x|^{-n-2s}.
\end{align*}

Now, if we allow a right hand side in the PDE, making $\delta > 0$, we obtain the following:
$$\left|b_R(x,t) - \int_{\R^n}p_t(x-y)b_0(y)\mathrm{d}y\right| \leq \delta\int_0^t\int_{\R^n}p_{t-\zeta}(x-y)\max\{|y|,1\}^{-n-2s}\mathrm{d}y\mathrm{d}\zeta,$$
and then we can estimate the second integral as follows. First we separate the integral in pieces, taking into account that $p_t(x-y) \lesssim \min\{t^{-\frac{n}{2s}},t|x-y|^{-n-2s}\}$, and also that $|x| \geq 2$.

\begin{align*}
    I_1 &:= \int_{B_1}t|x-y|^{-n-2s}\mathrm{d}y \lesssim t|x|^{-n-2s},\\
    I_2 &:= \int_{B_{t^{\frac{1}{2s}}}(x)}t^{-\frac{n}{2s}}\max\{1,|y|\}^{-n-2s}\mathrm{d}y \lesssim (t^{\frac{1}{2s}})^nt^{-\frac{n}{2s}}|x|^{-n-2s} = |x|^{-n-2s},\\
    I_3 &:= \int_{B_1(x)\setminus B_{t^{\frac{1}{2s}}}(x)}t|x-y|^{-n-2s}|y|^{-n-2s}\mathrm{d}y \lesssim t(t^{\frac{1}{2s}})^{-2s}|x|^{-n-2s} = |x|^{-n-2s},\\
    I_4 &:= \int_{B_1^c\cap B_1^c(x)}t|x-y|^{-n-2s}|y|^{-n-2s}\mathrm{d}y = t\int_{B_1^c\cap B_1^c(x)}|x-y|^{-n-2s}|y|^{-n-2s}\mathrm{d}y\\
    &= 2t\int_{B_1^c\cap\{x\cdot y \leq |x|^2/2\}}|x-y|^{-n-2s}|y|^{-n-2s}\mathrm{d}y \lesssim t|x|^{-n-2s}\int_{B_1^c}|y|^{-n-2s}\mathrm{d}y \lesssim t|x|^{-n-2s},
\end{align*}
where we used that $|x-y| \geq \frac{|x|}{2}$ in the half-space $\{x\cdot y \leq |x|^2/2\}$ to estimate $I_4$.

Putting everything together, we have
$$\int_{\R^n}p_t(x-y)\max\{|y|,R\}^{-n-2s}\mathrm{d}y \leq I_1 + I_2 + I_3 + I_4 \lesssim |x|^{-n-2s}.$$

Therefore, the error term introduced by the right hand side in the PDE can be bounded by the main term:
\begin{align*}
    \left|b_R(x,t) - \int_{\R^n}p_t(x-y)b_0(y)\mathrm{d}y\right| &\leq \delta \int_0^t\int_{\R^n}p_{t-\zeta}(x-y)\max\{|y|,1\}^{-n-2s}\mathrm{d}y\mathrm{d}\zeta\\
    &\lesssim \delta t|x|^{-n-2s} \lesssim \delta\int_{\R^n}p_t(x-y)b_0(y)\mathrm{d}y.
\end{align*}

Thus, choosing $\delta$ small enough, we have $b_R(x,t) \asymp t|x|^{-n-2s}$ for $|x| \geq 2$.
\end{proof}

\section{The penalized parabolic obstacle problem}\label{sect:appendix_penaliz}
First, we need that the penalized problem has a unique solution. To do that, we first prove that there holds a comparison principle.

\begin{lem}\label{lem:comparison_beta}
Let $\varepsilon > 0$, let $L$ be a nonlocal operator satisfying (\ref{eq:operator}) and (\ref{eq:operator_elliptic}), and let $f, g, \varphi, \psi, u_0$ and $v_0$ be uniformly Lipschitz and bounded, and let $u$ and $v$ be uniformly Lipschitz and bounded solutions of the following parabolic problems:
$$\left\{\begin{array}{rclll}
\p_tu + Lu & = & \beta_\varepsilon(u - \varphi) + f & \text{in} & \R^n\times(0,T)\\
u(\cdot,0) & = & u_0,
\end{array}\right.$$
$$\left\{\begin{array}{rclll}
\p_tv + Lv & = & \beta_\varepsilon(v - \psi) + g & \text{in} & \R^n\times(0,T)\\
v(\cdot,0) & = & v_0,
\end{array}\right.$$
where $\beta_\varepsilon(z) = e^{-z/\varepsilon}$.
Assume additionally that $u_0 \leq v_0$, $\varphi \leq \psi$ and $f \leq g$. Then, $u \leq v$ in $\R^n\times(0,T)$.
\end{lem}

\begin{proof}
Assume that $\inf (v - u) < 0$, otherwise there is nothing to prove. Let $\delta > 0$ small, $M > 0$ large to be chosen later, and let $p(x) = (1+|x|)^s$. First, one can check by a direct computation that $Lp$ is bounded. Then, the function
$$w(x,t) = v(x,t) - u(x,t) + \frac{\delta}{T - t} + \delta p(x) + \delta M$$
has an absolute minimum in $\R^n\times[0,T]$, and taking $\delta$ small enough, the minimum is negative. Let $(x_0,t_0)$ be the minimum point. First, observe that, since the minimum is negative, $t_0 > 0$, because $v \geq u$ at $t = 0$. Notice also that $t_0 < T$ because $\delta(T-t)^{-1}$ tends to infinity as $t \rightarrow T$. Then, $(x_0,t_0)$ is an interior point and then we can differentiate in $t$ and evaluate $L$, which is well defined thanks to the uniform Lipschitz regularity. Therefore,
\begin{gather*}
    v_t(x_0,t_0) - u_t(x_0,t_0) + \frac{\delta}{(T-t_0)^2} = 0\\
    Lv(x_0,t_0) - Lu(x_0,t_0) + \delta Lp(x_0) \leq 0.
\end{gather*}
Furthermore, we can also evaluate the equations at $(x_0,t_0)$ to obtain
\begin{align*}
    (\p_t + L)u(x_0,t_0) &= \beta_\varepsilon(u(x_0,t_0) - \varphi(x_0)) + f(x_0,t_0)\\
    (\p_t + L)v(x_0,t_0) &= \beta_\varepsilon(v(x_0,t_0) - \psi(x_0)) + g(x_0,t_0).
\end{align*}
And then combining the equations and using that $\beta_\varepsilon$ is decreasing,
\begin{align*}
    \beta(v - \varphi) - \beta(u - \varphi) &\leq \beta(v - \psi) + g - \beta(u - \varphi) - f\\
    &= (\p_t+L)(v-u) \leq \delta\left[Lp - \frac{1}{(T-t_0)^2}\right] \leq C\delta,
\end{align*}
where we have omitted that all the functions are considered at the point $(x_0,t_0)$ for ease of read. It follows that $v(x_0,t_0) - u(x_0,t_0) \geq -C'\delta$. Therefore, choosing $M > C'$, $w(x_0,t_0) > 0$, a contradiction. Therefore $v \geq u$ in $\R^n\times(0,T)$.
\end{proof}

Then, using the Perron method, one can construct a viscosity solution for the penalized problem.

\begin{prop}\label{prop:classic_penaliz}
For all $\varepsilon > 0$ and $\varphi \in C^{2,1}_c(\R^n)$, there exists a unique viscosity solution, $u^\varepsilon \in C(\R^n\times[0,T])\cap L^\infty(\R^n\times[0,T])$, to the penalized problem
$$\left\{\begin{array}{rclll}
\p_tu^\varepsilon + Lu^\varepsilon & = & \beta_\varepsilon(u^\varepsilon - \varphi) & \text{in} & \R^n\times(0,T)\\
u^\varepsilon(\cdot,0) & = & \varphi + \sqrt{\varepsilon},
\end{array}\right.$$
where $\beta_\varepsilon(z) = e^{-z/\varepsilon}$.
\end{prop}

\begin{proof}[Sketch of the proof]
The proof follows the standard techniques in viscosity solutions, see \cite{IS12} for a detailed explanation in the case of local operators.

To see existence, we construct a bounded continuous subsolution and supersolution, and then we will take the infimum of all supersolutions as our solution.

It is easy to check that $u_-(x,t) = -\|\varphi\|_{L^\infty(\R^n)}$ is a subsolution. Indeed, 
$$u_-(\cdot,0) \leq \varphi + \sqrt{\varepsilon} \quad \text{and} \quad (\p_t+L)u_- - \beta_\varepsilon(u_- - \varphi) = -\beta_\varepsilon(u_- - \varphi) \leq 0.$$

On the other hand, $u_+(x,t) = \|\varphi\|_{L^\infty(\R^n)} + \sqrt{\varepsilon} + t$ is a supersolution. The initial condition is immediately fulfilled, and
$$(\p_t+L)u_+ - \beta_\varepsilon(u_+ - \varphi) = 1 - \beta_\varepsilon(u_+ - \varphi) \geq 1 - \beta_\varepsilon(\sqrt{\varepsilon}) = 1 - e^{-1/\sqrt{\varepsilon}} > 0.$$

Then, we can apply the standard procedure for viscosity solutions and define
$$u^*(x,t) := \inf \{u(x,t) | \ u \text{ is a supersolution} \},$$
and then it can be checked that $u^*$ is a solution in the viscosity sense. Furthermore, $u_- \leq u^* \leq u_+$.

By interior regularity, such solution $u^*$ is a classical solution, and thus uniqueness follows from Lemma \ref{lem:comparison_beta}.
\end{proof}

Then, we prove some basic properties of solutions to this problem. The following lemma is analogous to the first part of \cite[Lemma 3.3]{CF13} for our case, and the proof is very similar.

\begin{lem}\label{lem:sub_penaliz}
Let $L$ be an operator satisfying (\ref{eq:operator}) and (\ref{eq:operator_elliptic}), let $\varphi \in C^{2,1}_c(\R^n)$ and let $u^\varepsilon$ be the solution of (\ref{eq:penalized}).

Then,
$$\beta_\varepsilon(u^\varepsilon - \varphi) \leq \max\{1,\|L\varphi\|_{L^\infty(\R^n)}\}.$$
In particular,
$$u^\varepsilon - \varphi \geq -\varepsilon\ln^+\|L\varphi\|_{L^\infty(\R^n)}.$$
\end{lem}

\begin{proof}
If $u^\varepsilon \geq \varphi$ everywhere, then $\beta_\varepsilon \leq 1$ and there is nothing to prove. Assume then otherwise, i.e., $\inf\limits_{\R^n\times[0,T]}(u^\varepsilon - \varphi) < 0$.

Then, since $u^\varepsilon \in L^\infty(\R^n\times(0,T))$, if $p(x) = (1+|x|)^s$ as in Lemma \ref{lem:comparison_beta}, for any $\delta > 0$ the function
$$w = u^\varepsilon - \varphi + \frac{\delta}{T - t} + \delta p$$
has a minimum point $(x_\varepsilon^\delta, t_\varepsilon^\delta) \in \R^n\times[0,T]$. Furthermore, if $\delta$ is small enough, $w(x_\varepsilon^\delta, t_\varepsilon^\delta) < 0$, and it follows that $t_\varepsilon^\delta \in (0,T)$. Hence, since the point is interior and $u^\varepsilon$ is smooth, then $\p_t w(x_\varepsilon^\delta, t_\varepsilon^\delta) = 0$ and $Lw(x_\varepsilon^\delta, t_\varepsilon^\delta) \leq 0$, which combined with the penalized equation (\ref{eq:penalized}) yields
$$\beta_\varepsilon(u^\varepsilon - \varphi)(x_\varepsilon^\delta, t_\varepsilon^\delta) \leq L\varphi(x_\varepsilon^\delta) - \frac{\delta}{(T-t_\varepsilon^\delta)^2} -\delta Lp(x_\varepsilon^\delta) \leq \|L\varphi\|_{L^\infty(\R^n)} + C\delta.$$

Finally, since $\beta_\varepsilon$ is decreasing and $(u^\varepsilon - \varphi)(x_\varepsilon^\delta, t_\varepsilon^\delta) \rightarrow \inf\limits_{\R^n\times[0,T]}(u^\varepsilon - \varphi)$ as $\delta \rightarrow 0$, we obtain that
$$\sup\limits_{\R^n\times[0,T]}\beta_\varepsilon(u^\varepsilon-\varphi) \leq \|L\varphi\|_{L^\infty(\R^n)},$$
as wanted.
\end{proof}

We can also prove an upper bound for $u^\varepsilon$.

\begin{lem}\label{lem:super_penaliz}
Let $L$ be an operator satisfying (\ref{eq:operator}) and (\ref{eq:operator_elliptic}), let $\varphi \in C^{2,1}_c(\R^n)$ and let $u^\varepsilon$ be the solution of (\ref{eq:penalized}).

Then,
$$u^\varepsilon(\cdot, t) - \varphi \leq \sqrt{\varepsilon} +  2t\max\{1,\|L\varphi\|_{L^\infty(\R^n)}\}.$$
\end{lem}

\begin{proof}
First, let us compute
$$(\p_t+L)(u^\varepsilon - \varphi) = \beta_\varepsilon(u^\varepsilon - \varphi) - L\varphi \leq \max\{1,\|L\varphi\|_{L^\infty(\R^n)}\} + \|L\varphi\|_{L^\infty(\R^n)},$$
where we used Lemma \ref{lem:sub_penaliz} to estimate $\beta_\varepsilon$.

Therefore, if we define
$$w(x,t) = u^\varepsilon(x,t) - \varphi(x) - \sqrt{\varepsilon} - 2t\max\{1,\|L\varphi\|_{L^\infty(\R^n)}\},$$
we get that $w(\cdot,0) \equiv 0$ and that $w$ is a subsolution, $(\p_t + L)w \leq 0$, by construction, and then the comparison principle for classical solutions of the nonlocal parabolic equation yields the result.
\end{proof}

We also need to see that we can differentiate the penalized problem.
\begin{lem}\label{lem:diff_penaliz}
Let $L$ be an operator satisfying (\ref{eq:operator}) and (\ref{eq:operator_elliptic}), let $\varphi \in C^{2,1}_c(\R^n)$ and let $u^\varepsilon$ be the solution to the penalized problem (\ref{eq:penalized}). Then, given any unit vector $\nu \in \R^n\times\R$,
\begin{align*}
(\p_t+L)u^\varepsilon_\nu &= \beta_\varepsilon'(u^\varepsilon-\varphi)(u^\varepsilon_\nu - \varphi_\nu),\\
(\p_t+L)u^\varepsilon_{\nu\nu} &= \beta_\varepsilon'(u^\varepsilon-\varphi)(u^\varepsilon_{\nu\nu} - \varphi_{\nu\nu}) + \beta_\varepsilon''(u^\varepsilon-\varphi)(u^\varepsilon_\nu - \varphi_\nu)^2,\\
u^\varepsilon_t(\cdot,0) &= -L\varphi + e^{-1/\sqrt{\varepsilon}},\\
u^\varepsilon_{tt}(\cdot,0) &= L^2\varphi - \frac{1}{\varepsilon}e^{-1/\sqrt{\varepsilon}}(e^{-1/\sqrt{\varepsilon}} - L\varphi),
\end{align*}
where the two last expressions must be understood in the sense of the uniform limit as $t \rightarrow 0^+$.
\end{lem}

\begin{proof}
We will iterate Proposition \ref{prop:interior_alphabeta}. First, $u^\varepsilon \in L^\infty(\R^n\times(0,T))$ by Lemmas \ref{lem:sub_penaliz} and \ref{lem:super_penaliz}. Then, observe that $\beta_\varepsilon(u^\varepsilon - \varphi) \in L^\infty$ as well.

Let $W = W_x\times[t_1,t_2]$ be a compact cylinder in $\R^n\times(0,T)$. Then, by Proposition \ref{prop:interior_alphabeta} and a covering argument, $\|u^\varepsilon\|_{C^{2s-\delta}_xC^{1-\delta}_t(W)} \leq C$ for a small $\delta > 0$ to be chosen later. Since $W$ is arbitrary, 
$$u^\varepsilon \in C^{2s-\delta}_xC^{1-\delta}_t(\R^n\times(0,T)),$$
and, since the previous estimates where invariant with respect to translations in $x$,
$$\|u^\varepsilon\|_{C^{2s-\delta}_xC^{1-\delta}_t(\R^n\times[t_1,t_2])} \leq C(t_1,t_2),$$
for any $0 < t_1 < t_2 < T$.

Now, repeating the same argument $k$ times we obtain that
$$\|u^\varepsilon\|_{C^{3,2s-\delta}_xC^{k(1-\delta)}_t(\R^n\times[t_1,t_2])} \leq C(t_1,t_2),$$
provided that $k$ is large enough. The cap in the $x$ regularity comes from the fact that $\varphi \in C^{2,1}_c$ and then $\beta_\varepsilon(u^\varepsilon - \varphi)$ cannot attain further regularity in $x$.

In particular, $u^\varepsilon \in C^3(\R^n\times(0,T))$, it is a classical solution, and then $u^\varepsilon_\nu$ and $u^\varepsilon_{\nu\nu}$ are at least $C^1$ in $\R^n\times(0,T)$, and they are also bounded for each $t \in (0,T)$, and therefore they are also classical solutions of their respective equations.

For the initial conditions, we recover the expression of $u^\varepsilon$ from Duhamel's formula,
$$u^\varepsilon = p_t * \varphi + \sqrt{\varepsilon} + \int_0^t p_\tau*\big(\beta(u^\varepsilon(\cdot,t-\tau)-\varphi)\big) \mathrm{d}\tau,$$
and then differentiate it with respect to $t$ to get
$$u^\varepsilon_t = \p_t p_t * \varphi + p_t * \left.\beta\right|_{t = 0} + \int_0^tp_\tau*\big(\beta'(u^\varepsilon(\cdot,t-\tau)-\varphi)u^\varepsilon_t(\cdot,t-\tau)\big)\mathrm{d}\tau.$$
Then, notice that $\p_t p_t = -L p_t$ because $p_t$ is a solution, and it follows that\linebreak $\p_tp_t*\varphi = p_t*(-L\varphi)$. Furthermore, $\beta(u^\varepsilon - \varphi) \equiv e^{-1/\sqrt{\varepsilon}}$ at $t = 0$, so putting everything together,
\begin{equation}\label{eq:u_penaliz_t_duhamel}
    u^\varepsilon_t = -p_t * (L\varphi) + e^{-1/\sqrt{\varepsilon}} -\frac{1}{\varepsilon}\int_0^tp_\tau * \big(\beta (u^\varepsilon(\cdot,t-\tau)-\varphi)u^\varepsilon_t(\cdot,t-\tau)\big)\mathrm{d}\tau.
\end{equation}

Since $p_t$ is an approximation to the identity (see Corollary \ref{cor:heat_kernel_aproxid}) and $\beta$ is bounded by Lemma \ref{lem:sub_penaliz}, taking the $L^\infty$ norm we can conclude that
$$\|u^\varepsilon_t(\cdot,t)\|_{L^\infty(\R^n)} \leq C_1 + C_2\int_0^t\|u^\varepsilon_t(\cdot,\tau)\|_{L^\infty(\R^n)}\mathrm{d}\tau,$$
which implies by the Gronwall inequality that $u^\varepsilon_t \in L^\infty(\R^n\times(0,t))$.

Then, again by (\ref{eq:u_penaliz_t_duhamel}), since $\beta$ and $u^\varepsilon_t$ are bounded, $L\varphi$ is uniformly $C^2$ and $p_t$ is an approximation to the identity, it follows that $u^\varepsilon_t \rightarrow -L\varphi + e^{-1/\sqrt{\varepsilon}}$ uniformly as $t \rightarrow 0^+$.

For the last identity, we first differentiate (\ref{eq:u_penaliz_t_duhamel}) with respect to time to obtain
\begin{equation}\label{eq:u_penaliz_tt_duhamel}
    u^\varepsilon_{tt} = -\p_tp_t*(L\varphi) - \frac{1}{\varepsilon}p_t*\left.(\beta u^\varepsilon_t)\right|_{t=0} - \frac{1}{\varepsilon}\int_0^tp_\tau*(\beta'(u^\varepsilon_t)^2 + \beta u^\varepsilon_{tt})(\cdot,t-\tau)\mathrm{d}\tau.
\end{equation}

Now, by the same arguments used to simplify (\ref{eq:u_penaliz_t_duhamel}),
$$u^\varepsilon_{tt} = p_t * \left(L^2\varphi -\frac{1}{\varepsilon}e^{-1/\sqrt{\varepsilon}}(-L\varphi+e^{-1/\sqrt{\varepsilon}})\right)- \frac{1}{\varepsilon}\int_0^tp_\tau*(\beta'(u^\varepsilon_t)^2 + \beta u^\varepsilon_{tt})(\cdot,t-\tau)\mathrm{d}\tau,$$
and then using the boundedness of $u^\varepsilon_t$ and a Gronwall inequality, analogously to what we did with $u^\varepsilon_t$, 
$$u^\varepsilon_{tt} \rightarrow L^2\varphi -\frac{1}{\varepsilon}e^{-1/\sqrt{\varepsilon}}(-L\varphi+e^{-1/\sqrt{\varepsilon}}),$$
uniformly as $t \rightarrow 0^+$.
\end{proof}

Finally, we prove that the solutions to the penalized problem converge to the solution to the obstacle problem.

\begin{proof}[Proof of Lemma \ref{lem:penalization}]
Let $\varepsilon \in (0,1)$.

\textit{Step 1.} First, recall the $L^\infty$ estimates for $u^\varepsilon - \varphi$. From Lemmas \ref{lem:sub_penaliz} and \ref{lem:super_penaliz},
$$-\varepsilon\ln^+\|L\varphi\|_{L^\infty(\R^n)} \leq u^\varepsilon - \varphi \leq \sqrt{\varepsilon}  + 2t\max\{1,\|L\varphi\|_{L^\infty(\R^n)}\}.$$

Now we use interior estimates and Arzelà-Ascoli to show that $u^\varepsilon \rightarrow u^0$ locally uniformly along a subsequence.

Let $W \subset\subset \R^n\times(0,T)$. Then, we can apply a version of \cite[Theorem 1.3]{FR17} to obtain
$$\|u^\varepsilon\|_{C^{1 - \delta}_t(W)} + \|u^\varepsilon\|_{C^{2s(1-\delta)}_x(W)} \leq C\big(\|u^\varepsilon\|_{L^\infty(\R^n\times(0,T))} + \|\beta_\varepsilon(u^\varepsilon - \varphi)\|_{L^\infty(\R^n\times(0,T))}\big) \leq C,$$
with $C$ only depending on $W$, $\|L\varphi\|_{L^\infty(\R^n)}$, $\delta > 0$, the dimension, $s$, and the ellipticity constants, because of the previous $L^\infty$ estimates on $u^\varepsilon$ and $\beta_\varepsilon$.

Hence, choosing a suitable small $\delta$, by the compact inclusion of Hölder spaces and Arzelà-Ascoli, $u^{\varepsilon_k} \rightarrow u^0$ uniformly in $W$ for some subsequence $\varepsilon_k \rightarrow 0$.

Now, consider a sequence of compact sets $W_0 \subset W_1 \subset \ldots$ such that their union is $\R^n\times(0,T)$ and repeat the same reasoning above. By a standard diagonalization argument, we can construct a sequence $\varepsilon_k$ such that $u^{\varepsilon_k} \rightarrow u^0$ locally uniformly in $\R^n\times(0,T)$.

\textit{Step 2.} Putting it together, we want to prove that, for all $\kappa > 0$, $u^{\varepsilon_k} \rightarrow u^0$ also in the ${L^\infty([0,T-\kappa]\rightarrow L^1_s)}$ norm. To do it, let $\tau > 0$ to be chosen later. Then, we distinguish two cases. If $t < \tau$,
\begin{align*}
\|u^{\varepsilon_k}(\cdot,t) - u^0(\cdot,t)\|_{L^1_s} &\leq \|u^{\varepsilon_k}(\cdot,t) - \varphi\|_{L^1_s} + \|\varphi - u^0(\cdot,t)\|_{L^1_s}\\
&\leq 2\sup\limits_{m \geq k}\|u^{\varepsilon_m} - \varphi\|_{L^1_s} \leq 2C\sup\limits_{m \geq k}\|u^{\varepsilon_m} - \varphi\|_{L^\infty(\R^n)}\\
&< 2C\big(\sqrt{\varepsilon_k} + 2\tau\max\{1,\|L\varphi\|_{L^\infty(\R^n)}\}\big).
\end{align*}

On the other hand, if $t \geq \tau$ we use the locally uniform convergence of the sequence. Let $R > 0$. Then, for all $t \in [\tau,T-\kappa]$,
\begin{align*}
\|u^{\varepsilon_k}(\cdot,t) - u^0(\cdot,t)\|_{L^1_s} &\lesssim \|u^{\varepsilon_k}(\cdot,t) - u^0(\cdot,t)\|_{L^\infty(B_R)} + R^{-2s}\|u^{\varepsilon_k}(\cdot,t) - u^0(\cdot,t)\|_{L^\infty(B_R^c)}\\
&\lesssim \|u^{\varepsilon_k}(\cdot,t) - u^0(\cdot,t)\|_{L^\infty(B_R)} + 2R^{-2s}\sup\limits_{m \geq k}\|u^{\varepsilon_m}(\cdot,t)\|_{L^\infty(B_R^c)}\\
&\lesssim \|u^{\varepsilon_k}(\cdot,t) - u^0(\cdot,t)\|_{L^\infty(B_R)} + R^{-2s},
\end{align*}
and then the second term tends to zero as $R \rightarrow \infty$ and then the first term tends to zero as $k$ goes to infinity  by the local uniform convergence.

Therefore, choosing first $\tau$ small, then $R$ big and then $k$ big, $\|u^{\varepsilon_k}(\cdot, t) - u^0(\cdot,t)\|_{L^1_s}$ can be made arbitrarily small, as we wanted to see.

\textit{Step 3.} Then we prove that $u^0$ is the solution of (\ref{eq:prev_problem}).

First, from the lower bound $u^{\varepsilon_k} \geq \varphi - {\varepsilon_k}\ln^+\|L\varphi\|_{L^\infty(\R^n)}$, taking the limit ${\varepsilon_k} \rightarrow 0$ it becomes clear that $u^0 \geq \varphi$. Then $(\p_t+L)u^{\varepsilon_k} = \beta_{\varepsilon_k}(u^{\varepsilon_k} - \varphi) \geq 0$, and the uniform limit of viscosity supersolutions is also a supersolution (with the extra convergence assumption of Step 2), by \cite[Theorem 5.3]{CD14}.

Hence, we only need to check that $(\p_t+L)u^0 = 0$ in the set $\{u^0(x,t) > \varphi(x)\}$ in the viscosity sense. Again by \cite[Theorem 5.3]{CD14}, it suffices to check the following.

Consider a compact set $E \subset \{u^0(x,t) > \varphi(x)\}$. By the uniform convergence of $u^{\varepsilon_k}$ to $u_0$, there exist $\mu > 0$ and $k_0$ such that for all $k \geq k_0$, $u^{\varepsilon_k}(x,t) > \varphi(x) + \mu$, for all $(x,t) \in E$. Hence,
$$(\p_t+L)u^{\varepsilon_k}(x,t) = \beta_{\varepsilon_k}(u^{\varepsilon_k} - \varphi)(x,t) \in (0,e^{-\mu / {\varepsilon_k}}),$$
and the limit of the right hand side is zero when ${\varepsilon_k} \rightarrow 0$.

Finally, from the $L^\infty$ estimates in Lemmas \ref{lem:sub_penaliz} and \ref{lem:super_penaliz}, it follows the concordance of the initial conditions, $u^0(\cdot,0) = \varphi$, and the continuity of $u^0$ as $t \rightarrow 0^+$.

\textit{Step 4.} Using the uniqueness of the solution we can eliminate the need to consider subsequences. Indeed, for any $\varepsilon_n \downarrow 0$, we can repeat Steps 2 and 3 to obtain a subsequence $u^{\varepsilon_{n_j}}$ that converges locally uniformly to the solution of (\ref{eq:prev_problem}). Therefore, $u^\varepsilon \rightarrow u^0$ locally uniformly as well.
\end{proof}

\end{document}